\documentclass[11pt]{amsart}

 \usepackage{amsfonts,graphics,amsmath,amsthm,amsfonts,amscd,amssymb,amsmath,latexsym,multicol,
 mathrsfs}
\usepackage{epsfig,url}
\usepackage{flafter}
\usepackage{fancyhdr}
\usepackage{tikz-cd}
\usepackage{hyperref}
\usepackage{comment}
\hypersetup{colorlinks=true, linkcolor=black}

 \usepackage[matrix, arrow]{xy}

\DeclareMathOperator{\Supp}{Supp}

\DeclareMathOperator{\Bir}{Bir}
\newcommand{\lf}{\left\lfloor}   
\newcommand{\rf}{\right\rfloor}

 \numberwithin{equation}{subsection}
 \numberwithin{footnote}{subsection}

 \newtheorem{lem}[subsection]{Lemma}
 \newtheorem{prop}[subsection]{Proposition}
 \newtheorem{thm}[subsection]{Theorem}
 \newtheorem{conj}[subsection]{Conjecture}

{
\theoremstyle{upright}

 \newtheorem{defn}[subsection]{Definition}

 \newtheorem{rem}[subsection]{Remark}

}

 \newcommand{\N}{\mathbb N}

 \newcommand{\Q}{\mathbb Q}
 \newcommand{\R}{\mathbb R}
 \newcommand{\Z}{\mathbb Z}
  \newcommand{\C}{\mathbb C}
  \newcommand{\Diff}{\text{Diff}}
 \newcommand{\bir}{\dashrightarrow}
 \newcommand{\rddown}[1]{\left\lfloor{#1}\right\rfloor} 

\title{\large C\MakeLowercase{omplements on Log Canonical} \large F\MakeLowercase{ano Varieties}}
\author{Yanning Xu}
\date{\today}
\begin{document}
\maketitle


\begin{abstract}
	In this paper, we generalise the theory of complements to log canonical log fano varieties and prove boundedness of complements for them in dimension less than or equal to $3$. We also prove some boundedness results for the canonical index of sdlt log Calabi Yau varieties in dimension 2.
\end{abstract}
\tableofcontents


\section{\bf Introduction}
The focus of this paper is to generalise the theory of complements to log canonical fano varieties. Strictly speaking, we have the following conjecture.
\begin{conj}\label{conj-main-compl}
	Let $d$ be a natural number and $\mathfrak{R}\subset [0,1]$ be a finite set of rational numbers.
	Then there exists a natural number $n$ 
	depending only on $d$ and $\mathfrak{R}$ satisfying the following.  
	Assume $(X,B)$ is a projective pair such that 
	\begin{itemize}
		
		\item $(X,B)$ is lc of dimension $d$,
		
		\item $B\in \Phi(\mathfrak{R})$, that is, the coefficients of $B$ are in $\Phi(\mathfrak{R})$, and 
		
		\item $-(K_X+B)$ is ample.
	\end{itemize}
	Then there is an $n$ complement $K_{X}+B^+$ of $K_X+B$ such that $B^+\geq B$.
\end{conj}
The main result of this paper is the following theorem.
\begin{thm}\label{thm-main_compl}
	Conjecture \ref{conj-main-compl} hold in dimension $d\leq 3$. 
\end{thm}
A key difference between log canonical fano varieties and fano type varieties is that we have to consider reducible varieties, or semi log canonical varieties and prove certain boundedness result on them. In this direction, we have the following conjecture.
\begin{conj}\label{conj-bound-sdlt-index}
Let $d$ be a natural number and $\mathfrak{R}\subset [0,1]$ be a finite set of rational numbers.
Then there exists a natural number $n$ 
depending only on $d$ and $\mathfrak{R}$ satisfying the following.  
Assume $(X,B)$ is a projective pair such that 
\begin{itemize}

\item $(X,B)$ is slc of dimension $d$,

\item $B\in \Phi(\mathfrak{R})$, that is, the coefficients of $B$ are in $\Phi(\mathfrak{R})$, and 

\item $K_X+B\sim_\Q 0$
\end{itemize}
Then $n(K_X+B)\sim 0$.

\end{conj}

This is related to the following well-known conjecture regarding the canonical index for lc Calabi Yau pairs. 

\begin{conj}\label{conj-bound-dlt-index}
	Let $d$ be a natural number and $\mathfrak{R}\subset [0,1]$ be a finite set of rational numbers.
	Then there exists a natural number $n$ 
	depending only on $d$ and $\mathfrak{R}$ satisfying the following.  
	Assume $(X,B)$ is a projective pair such that 
	\begin{itemize}
		
		\item $(X,B)$ is lc of dimension $d$,
		
		\item $B\in \Phi(\mathfrak{R})$, that is, the coefficients of $B$ are in $\Phi(\mathfrak{R})$, and 
		
		\item $K_X+B\sim_\Q 0$
	\end{itemize}
	Then  $n(K_X+B)\sim 0$.
	
\end{conj}

Note that it is known that Conjecture \ref{conj-bound-dlt-index} holds in dimension 2 and some cases of it hold in dimension 3. It is widely expected that it holds in all dimension. 

Hence we have the following definite fact. 

\begin{thm}\label{thm-bnd-dlt-index-2}
	Conjecture \ref{conj-bound-dlt-index} and \ref{conj-bound-sdlt-index} hold in dimension 2. 
\end{thm}

\section*{Acknowledgement}  I would like to express my sincere gratitude to my advisor Prof. Caucher Birkar for the continuous support of my Ph.D study and teaching me Birational Geometry and the theory of Complements. I would also like to thank Prof. Vyacheslav V. Shokurov for his valuable comments on this paper.
\section{\bf Preliminaries}

All the varieties in this paper are quasi-projective over a fixed algebraically closed field of characteristic zero and a divisor means an $\R$-divisor
unless stated otherwise. The set of natural numbers $\N$ is the set of positive integers.

\subsection{Contractions}
In this paper a \emph{contraction} refers to a projective morphism $f\colon X\to Y$ of varieties 
such that $f_*\mathcal{O}_X=\mathcal{O}_Y$ ($f$ is not necessarily birational). In particular, $f$ has connected fibres and 
if $X\to Z\to Y$ is the Stein factorisation of $f$, then $Z\to Y$ is an isomorphism. Moreover, 
if $X$ is normal, then $Y$ is also normal. A contraction $f$ is \emph{small} if $f$ does not contract any divisor. A birational map $\pi: X \dashrightarrow Y $ is a \emph{birational contraction} if the inverse of $\pi$ does not contract divisors. 

\subsection{Hyperstandard Sets}\label{ss-dcc-sets}
For a subset $V\subseteq \R$ and a number $a\in\R$, we define $V^{\ge a}=\{v\in V \mid v\ge a\}$. 
We similarly define $V^{\le a}, V^{<a}$, and $V^{>a}$.

Let $\mathfrak{R}$ be a subset of $[0,1]$. Following [\cite{PSh-II}, 3.2] we define 
$$
\Phi(\mathfrak{R})=\left\{1-\frac{r}{m} \mid r\in \mathfrak{R},~ m\in \N\right\}
$$
to be the set of \emph{hyperstandard multiplicities} associated to $\mathfrak{R}$. We usually assume  
$0,1\in \mathfrak{R}$ without mentioning, so $\Phi(\mathfrak{R})$ includes $\Phi({\{0,1\}})$, the set of usual 
\emph{standard multiplicities}. Note that if we add  
$1-r$ to $\mathfrak{R}$ for each $r\in\mathfrak{R}$, then we get
$\mathfrak{R}\subset \Phi(\mathfrak{R})$.

Now assume $\mathfrak{R}\subset [0,1]$ is a finite set of rational numbers. Then 
$\Phi(\mathfrak{R})$ is a DCC set of rational numbers whose only accumulation point is $1$. 
We define $I=I(\mathfrak{R})$ to be the smallest natural number so that $Ir\in \Z$ 
for every $r\in \mathfrak{R}$. If $n\in \N$ is divisible by $I(\mathfrak{R})$, 
then $nb\le \rddown{(n+1)b}$ for every $b\in \Phi(\mathfrak{R})$ [\cite{PSh-II}, Lemma 3.5].

\subsection{Divisors}
Let $X$ be a normal variety, and let $M$ be an $\R$-divisor on $X$. 
We denote the coefficient of a prime divisor $D$ in $M$ by $\mu_DM$. If every non-zero coefficient of 
$M$ belongs to a set $\Phi\subseteq \R$, we write $M\in \Phi$. Writing $M=\sum m_iM_i$ where 
$M_i$ are the distinct irreducible components, the notation $M^{\ge a}$ means 
$\sum_{m_i\ge a} m_iM_i$, that is, we ignore the components with coefficient $<a$. One similarly defines $M^{\le a}, M^{>a}$, and $M^{<a}$. 

We say $M$ is \emph{b-Cartier} if it is $\Q$-Cartier and if there is a birational contraction 
$\phi\colon W\to X$ from a normal variety such that $\phi^*M$ is Cartier.    

Now let $f\colon X\to Z$ be a morphism to a normal variety. We say $M$ is \emph{horizontal} over $Z$ 
if the induced map $\Supp M\to Z$ is dominant, otherwise we say $M$ is \emph{vertical} over $Z$. If $N$ is an $\R$-Cartier 
divisor on $Z$, we often denote $f^*N$ by $N|_X$. 

Again let $f\colon X\to Z$ be a morphism to a normal variety, and let $M$ and $L$ be $\R$-Cartier divisors on $X$. 
We say $M\sim L$ over $Z$ (resp. $M\sim_\Q L$ over $Z$)(resp. $M\sim_\R L$ over $Z$) if there is a Cartier  
(resp. $\Q$-Cartier)(resp. $\R$-Cartier) divisor $N$ on $Z$ such that $M-L\sim f^*N$  
(resp. $M-L\sim_\Q f^*N$)(resp. $M-L\sim_\R f^*N$). For a point $z\in Z$, we say $M\sim L$ over $z$ if   
$M\sim L$ over $Z$ perhaps after shrinking $Z$ around $z$. The properties $M\sim_\Q L$ and $M\sim_\R L$ over $z$ 
are similarly defined.

For a birational map $X\bir X'$ (resp. $X\bir X''$)(resp. $X\bir X'''$)(resp. $X\bir Y$) 
whose inverse does not contract divisors, and for 
an $\R$-divisor $M$ on $X$ we usually denote the pushdown of $M$ to $X'$ (resp. $X''$)(resp. $X'''$)(resp. $Y$) 
by $M'$ (resp. $M''$)(resp. $M'''$)(resp. $M_Y$).
\subsection{Divisorial Sheaves}
We will also introduce the notion of a divisorial sheaf. Let $X$ be an S2 scheme. A divisorial sheaf is a rank one reflexive sheaf. Note that if $L$ is a divisorial sheaf, then $L^{[m]} := (L^{m})^{**}$, since tensor powers of a reflexive sheaf may not be reflexive. We also note that if $X$ is a normal variety, then divisorial sheaves correspond one to one to Weil divisors on $X$ modulo linear equivalence, via a Weil divisor $D$ corresponding to the sheaf $\mathcal{O}_X(D)$.

\subsection{b-divisors}\label{ss-b-divisor}

We recall some definitions regarding b-divisors but not in full generality. 
Let $X$ be a variety. A \emph{b-$\R$-Cartier b-divisor over $X$} is the choice of  
a projective birational morphism 
$Y\to X$ from a normal variety and an $\R$-Cartier divisor $M_Y$ on $Y$ up to the following equivalence: 
another projective birational morphism $Y'\to X$ from a normal variety and an $\R$-Cartier divisor
$M_{Y'}$ defines the same b-$\R$-Cartier b-divisor if there is a common resolution $W\to Y$ and $W\to Y'$ 
on which the pullbacks of $M_Y$ and $M_{Y'}$ coincide.  

A b-$\R$-Cartier b-divisor represented by some $Y\to X$ and $M_Y$ is \emph{b-Cartier} if $M_Y$ is 
Cartier. We will denote it by $M$ the push-down of $M_Y$ on $X$ if there is no confusion. Similarly one defines a \emph{b-nef} b-divisor.

\begin{lem}\label{l-pull-back-rational-function}
	Let $f\colon X\to Z$ be a contraction between normal varieties and let $g$ be a rational function on $X$.
	Assume that $\mathrm{div} g$ is vertical over $Z$. Then $g$ can be decomposed as $g=f \circ h$ for some rational function $h$ on $Z$.
\end{lem}
\begin{proof}
	Since $\mathrm{div} g$ is vertical over $Z$, there is a non-empty open subset $U \subseteq Z$ such that $g$ is regular on $f^{-1} U$. Note that $f_* \mathcal{O}_X(f^{-1}U) = \mathcal{O}_U$ which in turn implies that $g|_{f^{-1}U}= f \circ h_U$ for some regular function $h_U$ on $U$. Hence the conclusion follows.
\end{proof}

\begin{lem}\label{l-pull-back-Cartier}
	Let $f\colon X\to Z$ be a contraction between normal varieties and let $M$ be a Cartier divisor on $X$.
	Assume that $M\sim_\Q 0$ over $Z$ and that $M$ is vertical over $Z$. Then, there is a b-Cartier divisor $M_Z$ such that $M=f^* M_Z$.
\end{lem}
\begin{proof}
	Pick a sufficiently ample divisor $H$ on $Z$. By adding $f^*H$ to $M$ we can assume $M$ is effective. Because $M\sim_\Q 0$, there is a natural number $m$, a rational function $g$ on $X$, and a $\Q$-divisor $D_Z$ such that 
	$$
	M+ \frac{1}{m} \mathrm{div} g =f^* D_Z.
	$$
	Since $M$ is vertical over $Z$, we deduce that $\mathrm{div} g$ is vertical over $Z$. By the previous lemma $g=f \circ h$ and it follows that
	$$
	M=  f^* (D_Z-\frac{1}{m} \mathrm{div} h).
	$$
	Letting $M_Z:=D_Z-\frac{1}{m} \mathrm{div} h$ we obtain the lemma.
\end{proof}

\subsection{Pairs}
A \emph{sub-pair} $(X,B)$ consists of a normal quasi-projective variety $X$ and an $\R$-divisor 
$B$ such that $K_X+B$ is $\R$-Cartier. 
If the coefficients of $B$ are at most $1$ we say $B$ is a 
\emph{sub-boundary}, and if in addition $B\ge 0$, 
we say $B$ is a \emph{boundary}. A sub-pair $(X,B)$ is called a \emph{pair} if $B\ge 0$ (we allow coefficients 
of $B$ to be larger than $1$ for practical reasons).

Let $\phi\colon W\to X$ be a log resolution of a sub-pair $(X,B)$. Let $K_W+B_W$ be the 
pulback of $K_X+B$. The \emph{log discrepancy} of a prime divisor $D$ on $W$ with respect to $(X,B)$ 
is $1-\mu_DB_W$ and it is denoted by $a(D,X,B)$.
We say $(X,B)$ is \emph{sub-lc} (resp. \emph{sub-klt})(resp. \emph{sub-$\epsilon$-lc}) 
if $a(D,X,B)$ $\ge 0$ (resp. $>0$)(resp. $\ge \epsilon$) for every $D$. When $(X,B)$ 
is a pair we remove the sub and say the pair is lc, etc. Note that if $(X,B)$ is a lc pair, then 
the coefficients of $B$ necessarily belong to $[0,1]$. Also if $(X,B)$ is $\epsilon$-lc, then 
automatically $\epsilon\le 1$ because $a(D,X,B)=1$ for most $D$. 

Let $(X,B)$ be a sub-pair. A \emph{non-klt place} of $(X,B)$ is a prime divisor $D$ on 
birational models of $X$ such that $a(D,X,B)\le 0$. A \emph{non-klt center} is the image on 
$X$ of a non-klt place. When $(X,B)$ is lc, a non-klt center is also called an 
\emph{lc center}.

\subsection{Minimal Model Program (MMP)}\label{ss-MMP} 
We will use standard results of the minimal model program. 
Assume $(X,B)$ is a pair and $X\to Z$ is a projective morphism. 
Assume $H$ is an ample$/Z$ $\R$-divisor and that $K_X+B+H$ is nef$/Z$. Suppose $(X,B)$ is klt or 
that it is $\Q$-factorial dlt. We can run an MMP$/Z$ on $K_X+B$ with scaling of $H$. If 
$(X,B)$ is klt and if either $K_X+B$ or $B$ is big$/Z$, then the MMP terminates [\cite{BCHM}]. If $(X,B)$ 
is $\Q$-factorial dlt, then in general we do not know whether the MMP terminates but 
we know that in some step of the MMP we reach a model $Y$ on which $K_Y+B_Y$, 
the pushdown of $K_X+B$, is a limit of movable$/Z$ $\R$-divisors: indeed, if the MMP terminates, then 
the claim is obvious; otherwise the MMP produces an infinite sequence $X_i\bir X_{i+1}$ 
of flips and a decreasing sequence $\lambda_i$ of numbers in $(0,1]$ such that 
$K_{X_i}+B_i+\lambda_iH_i$ is nef$/Z$; by [\cite{BCHM}][\cite{B-lc-flips}, Theorem 1.9], $\lim\lambda_i=0$; 
in particular, if $Y:=X_1$, then $K_Y+B_Y$ is the limit of the movable$/Z$ $\R$-divisors 
$K_Y+B_Y+\lambda_i H_Y$.

\subsection{Generalized polarised pairs}\label{ss-gpp}
For the basic theory of generalized polarised pairs see [\cite{BZh}, Section 4].
Below we recall some of the main notions and discuss some basic properties.\\

(1)
A \emph{generalized polarised pair} consists of 
\begin{itemize}
	\item a normal variety $X'$ equipped with a projective
	morphism $X'\to Z$, 
	
	\item an $\R$-divisor $B'\ge 0$ on $X'$, and 
	
	\item a b-$\R$-Cartier  b-divisor over $X'$ represented 
	by some projective birational morphism $X \overset{\phi}\to X'$ and $\R$-Cartier divisor
	$M$ on $X$
\end{itemize}
such that $M$ is nef$/Z$ and $K_{X'}+B'+M'$ is $\R$-Cartier,
where $M' := \phi_*M$. 

We usually refer to the pair by saying $(X',B'+M')$ is a  generalized pair with 
data $X\overset{\phi}\to X'\to Z$ and $M$. Since a b-$\R$-Cartier b-divisor is defined birationally (see \ref{ss-b-divisor}), 
in practice we will often replace $X$ with a resolution and replace $M$ with its pullback.
When $Z$ is not relevant we usually drop it
and do not mention it: in this case one can just assume $X'\to Z$ is the identity. 
When $Z$ is a point we also drop it but say the pair is projective. 

Now we define generalized singularities.
Replacing $X$ we can assume $\phi$ is a log resolution of $(X',B')$. We can write 
$$
K_X+B+M=\phi^*(K_{X'}+B'+M')
$$
for some uniquely determined $B$. For a prime divisor $D$ on $X$ the \emph{generalized log discrepancy} 
$a(D,X',B'+M')$ is defined to be $1-\mu_DB$. 

We say $(X',B'+M')$ is 
\emph{generalized lc} (resp. \emph{generalized klt})(resp. \emph{generalized $\epsilon$-lc}) 
if for each $D$ the generalized log discrepancy $a(D,X',B'+M')$ is $\ge 0$ (resp. $>0$)(resp. $\ge \epsilon$).
We say $(X',B'+M')$ is \emph{generalized dlt} if it is generalized lc, $(X',B')$ is dlt, and every generalized non-klt center 
of $(X',B'+M')$ is a non-klt center of $(X',B')$ (note that here we are assuming $(X',B')$ is a 
dlt pair in the usual sense, in particular, $K_{X'}+B'$ is assumed to be $\R$-Cartier). If in addition the connected 
components of $\rddown{B'}$ are irreducible, we say the pair is \emph{generalized plt}.

A \emph{generalized non-klt center} of a generalized pair $(X',B'+M')$ is the image of a prime divisor 
$D$ on birational models of $X'$ with $a(D,X',B'+M')\le 0$, and 
the \emph{generalized non-klt locus} of the generalized pair is the union of all the generalized non-klt centers.

(2)
Let $(X',B'+M')$ be a generalized pair as in (1) and let $\psi\colon X''\to X'$ be a projective birational 
morphism from a normal variety. Replacing $\phi$ we can assume $\phi$ factors through 
$\psi$. We then let $B''$ and $M''$ be the pushdowns of 
$B$ and $M$ on $X''$ respectively. In particular,  
$$
K_{X''}+B''+M''=\psi^*(K_{X'}+B'+M').
$$
If $B''\ge 0$, then $(X'',B''+M'')$ is also a generalized pair 
with data $X\to X''\to Z$ and $M$. If $(X'',B''+M'')$ is $\Q$-factorial generalized dlt and if 
every exceptional prime divisor of $\psi$ appears in $B''$ with coefficients one, then we say 
$(X'',B''+M'')$ is a \emph{$\Q$-factorial generalized dlt model(or blow-up)} of $(X',B'+M')$. Such models exist if 
$(X',B'+M')$ is generalized lc, by [\cite{BZh}, Lemma 4.5].

\begin{lem}[Connectedness principle]\label{l-connectedness}
	Let $(X',B'+M')$ be a generalized pair with data $X\overset{\phi}\to X'\to Z$ and $M$ 
	where $X'\to Z$ is a contraction. 
	Assume $-(K_{X'}+B'+M')$ is nef and big over $Z$. Then the generalized non-klt locus of 
	$(X',B'+M')$ is connected near each fibre of $X'\to Z$.
\end{lem}

\subsection{Canonical bundle formula}
An algebraic fiber space $f:X \rightarrow Z$ with a given log canonical divisor $K_X+B$ which is $\Q$-linearly trivial over $Z$ is called an \emph{lc-trivial fibration}. The following canonical bundle formula for lc-trivial fibration from [\cite{Amb}] or [\cite{Fujino-Gongyo}] will be used.  

\begin{thm}[\text{[\cite{Fujino-Gongyo}, Theorem 1.1], [\cite{Amb}, Theorem 3.3]}]\label{c-bdle-form}
	Let $f:(X,B) \rightarrow Z$ be a projective morphism from a log pair to a normal variety $Z$ with connected fibers, $B$ be a $\mathbb{Q}$-boundary divisor. Assume that $(X,B)$ is sub-lc and it is lc on the generic fiber and $K_X+B \sim_\Q 0/Z$. Then there exists a boundary $\mathbb{Q}$-divisor $B_Z$ and a $\mathbb{Q}$-divisor $M_Z$ on $Z$ satisfying the following properties.  
	
	(i). $(Z,B_Z+M_Z)$ is a generalized pair, 
	
	(ii). $M_Z$ is b-nef.
	
	(iii). $K_X+B \sim_\mathbb{Q} f^\ast (K_Z+B_Z+M_Z)$.
\end{thm}

Under the above notations, $B_Z$ is called the \emph{discriminant part} and $M_Z$ is the \emph{moduli part}. We add a few words on the construction of $B_Z$. For each prime divisor $D$ on $Z$ we let 
$t_D$ be the lc threshold of $f^*D$ with respect to $(X,B)$ over the generic point of $D$, that is, 
$t_D$ is the largest number so that $(X,B+t_Df^*D)$ is sub-lc over the generic point of $D$. 
Of course $f^*D$ may not be well-defined everywhere but at least it is defined over the 
smooth locus of $Z$, in particular, near the generic point of $D$, and that is all we need. 
Next let $b_D=1-t_D$, and then define $B_Z=\sum b_DD$ where the sum runs over all the 
prime divisors on $Z$. 

\subsection{Slc Pairs}
We will introduce the basic definition here. More details of slc and sdlt pairs will be introduced in more details later. First we introduce demi normal schemes as in \cite{Kol}. A \emph{demi-normal scheme} is a scheme that is S2 and normal crossing in codimensional 1. Let $\Delta$ be an effective
$\Q$-divisor whose support does not contain any irreducible components of the conductor of X.
The pair $(X, \Delta)$ is called a \emph{semi-log canonical pair} (an \emph{slc pair}, for short) if
\begin{itemize}
	\item $K_X + \Delta$ is $\Q$-Cartier, and
	
	\item $(X', \Theta)$ is log canonical, where $\pi : X' \rightarrow X$ is the normalization and $K_{X'}+ \Theta = \pi^*(K_X +\Delta)$.
\end{itemize}
\begin{rem}
	\begin{enumerate}
		\item  Note that $\Theta = \Delta' +D'$, where $D'$ is the conductor on the normalisation and $\Delta'$ is the divisorial part of the preimage of $\Delta$ on $X'$. For more detailed treatment, see \cite{Kol}, Chapter 5. 
		\item 	 We also note that if $X := \cup X_i$ where $X_i$ are irreducible components of $X$, and let $X_i' \rightarrow X_i$ be their normalisations, then we have $X' = \sqcup X_i'$.
	\end{enumerate}
\end{rem}

Similarly one defines a \emph{semi-divisorial log terminal pair} (an \emph{sdlt pair}, for short). An slc pair $(X,B)$ is said to be \emph{sdlt} if the normalisation $(X',\Theta)$ is dlt in the usual sense and $\pi: X_i'\rightarrow X_i$ is isomorphism. Note that here we are using the definition in \cite{Fuj1} instead of \cite{Kol}. This is fine since we will only use semi dlt in the following setting. We remark that if $(X,B)$ is a usual dlt pair, then $(\rddown{B},\Diff(B-\rddown{B}))$ is semi-dlt. Also for sdlt pair $(X,B)$ it is clear that $(K_X+B)|_{X_i} = K_{X_i}+\Theta_i$, where $\Theta_i := \Theta|_{X_i}$. We also have the following lemma.
\begin{lem}[\cite{Fujino-Gongyo-1}, Example 2.6]\label{lem-lt-slc-adjunction}
	Let $(X,B)$ be $\Q$-factorial lc with $B$ a $\Q$-divisor. Let $S := \rddown{B}$ and assume  $(X,B-\epsilon S)$ is klt for $0<\epsilon<<1$. Then $(S, \Diff(B-S))$ is slc.
\end{lem}

\subsection{Theorems on ACC}
Theorems on ACC proved in [\cite{HMX}] are crucial to the argument of the main results.

\begin{thm}[ACC for log canonical thresholds, \text{[\cite{HMX}, Theorem ~1.1]}]\label{ACC-1}
	Fix a positive integer $n$, a set $I \subset [0, 1]$ and a set $J \subset \R_{>0}$, where
	$I$ and $J$ satisfy the DCC. Let $\mathfrak{T}_n(I)$ be the set of log canonical pairs
	$(X, \Delta)$, where $X$ is a variety of dimension $n$ and the coefficients of $\Delta$
	belong to $I$. Then the set
	$$
	\{\mathrm{lct}(X, \Delta; M) | (X, \Delta) \in \mathfrak{T}_n(I), \mathrm{~the ~coefficients ~of~} M \mathrm{~ belong~ to~} J\}
	$$
	satisfies the ACC.
\end{thm}

\begin{thm}[ACC for numerically trivial pairs, \text{[\cite{HMX}, Theorem D]}]\label{ACC-2}
	Fix a positive integer $n$ and a set $I \subset [0, 1]$, which satisfies the
	DCC.
	Then there is a finite set $I_0 \subset I$ with the following property: \\
	If $(X, \Delta)$ is a log canonical pair such that \\
	(i) $X$ is projective of dimension $n$, \\
	(ii) the coefficients of $\Delta$ belong to $I$, and \\
	(iii) $K_X + \Delta$ is numerically trivial, \\
	then the coefficients of $\Delta$ belong to $I_0$.	
\end{thm}

The theorems above can be extended to the setting of generalized polairzed pairs.

Assume that $D'$ on $X'$ is an effective $\R$-divisor and that $N$ on $X$ is an $\R$-divisor which is nef and that $D' + N'$ is $\R$-Cartier. The generalized lc threshold of $D' + N'$ with respect to
$(X', B' + M')$ is defined as
$$
\sup \{s | \text{$(X'
	, B' + sD' + M' + sN')$ is generalized lc.} \}
$$
where the pair in the definition has boundary part $B'+sD'$ and nef part $M+sN$.

\begin{thm}[\text{[\cite{BZh}, Theorem 1.5]}]\label{ACC3}
	Let $\Lambda$ be a DCC set of nonnegative real numbers and $d$ a natural
	number. Then there is an ACC set $\Theta$ depending only on $\Lambda$, $d$ such that if
	$(X', B' + M')$, $M$, $N$, and $D'$ are as above satisfying
	
	(i) $(X', B' + M')$ is generalized lc of dimension $d$,
	
	(ii) $M = \sum \mu_j M_j$ where $M_j$ are nef Cartier divisors and $\mu_j \in \Lambda$,
	
	(iii) $N =\sum \nu_k N_k$ where $N_k$ are nef Cartier divisors and $\nu_k \in \Lambda$, and
	
	(iv) the coefficients of $B'$ and $D'$
	belong to $\Lambda$,
	
	then the generalized lc threshold of $D' +N'$ with respect to $(X', B' +M')$ belongs
	to $\Theta$.
\end{thm}

\begin{thm}[\text{[\cite{BZh}, Theorem 1.6]}]\label{ACC4}
	Let $\Lambda$ be a DCC set of nonnegative real numbers and $d$ a natural
	number. Then there is a finite subset $\Lambda_0 \subset \Lambda$ depending only on $\Lambda$, $d$ such that
	if $(X', B' + M')$ and $M$ are as above satisfying
	
	(i) $(X', B' + M')$ is generalized lc of dimension $d$,
	
	(ii) $M = \sum \mu_j M_j$ where $M_j$ are nef Cartier divisors and $\mu_j \in \Lambda$,
	
	(iii) $\mu_j = 0$ if $M_j \equiv 0$,
	
	(iv) the coefficients of $B'$
	belong to $\Lambda$, and
	
	(v) $K_{X'} + B' + M' \equiv 0$,
	
	then the coefficients of $B'$ and the $\mu_j$ belong to $\Lambda_0$.
\end{thm}

\subsection{Complements}\label{ss-compl}
(1) 
We first recall the definition for usual pairs.
Let $(X,B)$ be a pair where $B$ is a boundary and let $X\to Z$ be a contraction. 
Let $T=\rddown{B}$ and $\Delta=B-T$. 
An \emph{$n$-complement} of $K_{X}+B$ over a point $z\in Z$ is of the form 
$K_{X}+{B}^+$ such that over some neighbourhood of $z$ we have the following properties:
\begin{itemize}
	\item $(X,{B}^+)$ is lc, 
	
	\item $n(K_{X}+{B}^+)\sim 0$, and 
	
	\item $n{B}^+\ge nT+\rddown{(n+1)\Delta}$.\\
\end{itemize}
From the definition one sees that 
$$
-nK_{X}-nT-\rddown{(n+1)\Delta}\sim n{B}^+-nT-\rddown{(n+1)\Delta}\ge 0
$$
over some neighbourhood of $z$ which in particular means the linear system 
$$
|-nK_{X}-nT-\rddown{(n+1)\Delta}|
$$
is not empty over $z$. Moreover, if $B^+\ge B$, then $-n(K_X+B)\sim n(B^+-B)$ over $z$, hence 
$|-n(K_X+B)|$ is non-empty over $z$.

(2) 
Now let $(X',B'+M')$ be a projective generalized pair with data $\phi\colon X\to X'$ 
and $M$ where $B'\in [0,1]$. 
Let $T'=\rddown{B'}$ and $\Delta'=B'-T'$. 
An \emph{$n$-complement} of $K_{X'}+B'+M'$ is of the form $K_{X'}+{B'}^++M'$ where 
\begin{itemize}
	\item $(X',{B'}^++M')$ is generalized lc,
	
	\item $n(K_{X'}+{B'}^++M')\sim 0$ and $nM$ is b-Cartier, and 
	
	\item $n{B'}^+\ge nT'+\rddown{(n+1)\Delta'}$.\\
\end{itemize}
From the definition one sees that 
$$
-nK_{X'}-nT'-\rddown{(n+1)\Delta'}-nM'\sim n{B'}^+-nT'-\rddown{(n+1)\Delta'}\ge 0
$$
which in particular means the linear system 
$$
|-nK_{X'}-nT'-\rddown{(n+1)\Delta'}-nM'|
$$
is not empty. Moreover, if $B'^+\ge B'$, then $-n(K_{X'}+B'+M')\sim n(B'^+-B')$, hence 
$|-n(K_{X'}+B'+M')|$ is non-empty.

(3) 
Let $(X,B)$ be a projecive slc pair.
Let $T=\rddown{B}$ and $\Delta=B-T$. 
An \emph{$n$-complement} of $K_{X}+B$ is of the form 
$K_{X}+{B}^+$ such that we have the following properties:
\begin{itemize}
	\item $(X,{B}^+)$ is slc, 
	
	\item $n(K_{X}+{B}^+)\sim 0$, and 
	
	\item $n{B}^+\ge nT+\rddown{(n+1)\Delta}$.\\
\end{itemize}

\subsection{Fano Varieties}
Let $(X,B)$ be a pair and $X\to Z$ a contraction. We say $(X,B)$ is \emph{log Fano} over $Z$ 
if it is lc and $-(K_X+B)$ is ample over $Z$; if $B=0$ 
we just say $X$ is Fano over $Z$. The pair is called \emph{weak log Fano} over $Z$ if it is lc 
and $-(K_X+B)$ is nef and big over $Z$; 
if $B=0$ we say $X$ is \emph{weak Fano} over $Z$.
We say $X$ is \emph{of Fano type} over $Z$ if $(X,B)$ is klt weak log Fano over $Z$ for some choice of $B$;
it is easy to see this is equivalent to existence of a big$/Z$ $\Q$-boundary (resp. $\R$-boundary) 
$\Gamma$ so that $(X,\Gamma)$ is klt and $K_X+\Gamma \sim_\Q 0/Z$ (resp. $\sim_\R$ instead of $\sim_\Q$). We say $X$ is \emph{of log Fano type} over $Z$ if $(X,B)$ is log Fano over $Z$ for some choice of $B$, and we say $X$ is \emph{of Calabi-Yau type} over $Z$ if $(X,B)$ is lc and $K_X+B \sim_\Q 0/Z$ (resp. $\sim_\R$ instead of $\sim_\Q$).

Assume $X$ is of Fano type over $Z$. Then we can run the MMP 
over $Z$ on any $\R$-Cartier $\R$-divisor $D$ on $X$ which ends with some model $Y$ [\cite{BCHM}]. If $D_Y$ is nef over $Z$,
we call $Y$ a \emph{minimal model} over $Z$ for $D$. If $D_Y$ is not nef$/Z$, then 
there is a $D_Y$-negative extremal contraction $Y\to T/Z$ with $\dim Y>\dim T$ and we call 
$Y$ a \emph{Mori fiber space} over $Z$ for $D$. 

The following theorems, as referred to as the boundedness of complements, were proved by Birkar [\cite{B-Fano}] which is conjectured by Shokurov [\cite{shokurov-surf-comp}, Conjecture 1.3] who proved it in 
dimension $2$ [\cite{shokurov-surf-comp}, Theorem 1.4] (see also [\cite{PSh-II}, Corollary 1.8], [\cite{PSh-I}, Theorem 3.1] and [\cite{Shokurov-log-flips}] for some cases). 

\begin{thm}[\text{[\cite{B-Fano}, Theorem 1.7]}]\label{t-bnd-compl-usual}
	Let $d$ be a natural number and $\mathfrak{R}\subset [0,1]$ be a finite set of rational numbers.
	Then there exists a natural number $n$ 
	depending only on $d$ and $\mathfrak{R}$ satisfying the following.  
	Assume $(X,B)$ is a projective pair such that 
	\begin{itemize}
		
		\item $(X,B)$ is lc of dimension $d$,
		
		\item $B\in \Phi(\mathfrak{R})$, that is, the coefficients of $B$ are in $\Phi(\mathfrak{R})$, 
		
		\item $X$ is of Fano type, and 
		
		\item $-(K_{X}+B)$ is nef.\\
	\end{itemize}
	Then there is an $n$-complement $K_{X}+{B}^+$ of $K_{X}+{B}$ 
	such that ${B}^+\ge B$. Moreover, the complement is also an $mn$-complement for any $m\in \N$. 
	
\end{thm}

\begin{thm}[\text{[\cite{B-Fano}, Theorem 1.8]}]\label{t-bnd-compl-usual-local}
	Let $d$ be a natural number and $\mathfrak{R}\subset [0,1]$ be a finite set of rational numbers.
	Then there exists a natural number $n$
	depending only on $d$ and $\mathfrak{R}$ satisfying the following.  
	Assume $(X,B)$ is a pair and $X\to Z$ is a contraction such that 
	\begin{itemize}
		\item $(X,B)$ is lc of dimension $d$ and $\dim Z>0$,
		
		\item $B\in \Phi(\mathfrak{R})$, 
		
		\item $X$ is of Fano type over $Z$, and 
		
		\item $-(K_{X}+B)$ is nef over $Z$.\\
	\end{itemize}
	Then for any point $z\in Z$, there is an $n$-complement $K_{X}+{B}^+$ of $K_{X}+{B}$ 
	over $z$ such that ${B}^+\ge B$. 
\end{thm}

\begin{thm}[\text{[\cite{B-Fano}, Theorem 1.10]}]\label{t-bnd-compl}
	Let $d$ and $p$ be natural numbers and $\mathfrak{R}\subset [0,1]$ be a finite set of rational numbers.
	Then there exists a natural number $n$ 
	depending only on $d,p$, and $\mathfrak{R}$ satisfying the following. 
	Assume $(X',B'+M')$ is a projective generalized polarised pair with data $\phi\colon X\to X'$ and $M$ 
	such that 
	\begin{itemize}
		\item $(X',B'+M')$ is generalized lc of dimension $d$,
		
		\item $B'\in \Phi(\mathfrak{R})$ and $pM$ is b-Cartier,
		
		\item $X'$ is of Fano type, and 
		
		\item $-(K_{X'}+B'+M')$ is nef.\\
	\end{itemize}
	Then there is an $n$-complement $K_{X'}+{B'}^++M'$ of $K_{X'}+{B'}+M'$ 
	such that ${B'}^+\ge B'$. 
\end{thm}

\subsection{B-Birational Maps and B-pluricanonical Representations}
We intoduce the notion of \emph{B-birational} as in \cite{Fuj1}. Let $(X,B)$, $(X',B')$ be  sub pairs, we say $f: (X,B)\dashrightarrow (X',B')$ is \emph{B-birational} if there is a common resolution $\alpha: (Y,B_Y)\rightarrow (X,B)$, $\beta:(Y,B_Y)\rightarrow (X,B)$ such that $K_Y+B_Y=\alpha^*(K_X+B)=\beta^*(K_{X'}+B')$ and a commuting diagram as the following.
\[
\begin{tikzcd}
&(Y,B_Y)\arrow[ld,"\alpha"'] \arrow[rd,"\beta"]&\\
(X,B)\arrow[rr,dashed,"f"]&&(X',B')
\end{tikzcd}
\] Let $$\Bir(X,B):=\{f|f:(X,B)\dashrightarrow (X,B) \text{\;is B-birational}\}$$ Let $n$ be a positive integer such that  $n(K_X+B)$ is Cartier. Then we define $$\rho_n: \Bir(X,B)\rightarrow Aut(H^0(X,n(K_X+B)))$$ be the representaion of the natural action of $\Bir(X,B)$ acting on $H^0(X,n(K_X+B))$ by pullbacking back sections.

We have the following theorem on B-birational representations.
\begin{prop} [\cite{Fujino-Gongyo-1} ]\label{prop-finiteness-bir-representation}
	Let $(X,B)$ be a projective (not necessarily connected) dlt pair with $n(K_X+B)\sim 0$ and $n$ is even, then $\rho_n(\Bir(X,B)) $ is finite. \qed
\end{prop}

\newpage
\section{\bf Adjunctions and Hyperstandard Sets}

In this section we review several kinds of adjunction and prove some adjunction formulae, especially for surfaces, which will be needed in the subsequent sections. In general adjunction is relating the (log) canonical divisors of two varieties that are somehow related. We are particularly interested in how the (hyperstandard) coefficients of the boundaries are related.

\subsection{Divisorial Adjunction}
We briefly review [\cite{B-Fano}, \S 3.1]. We consider adjunction for a prime divisor on a variety.\\
(1) 
Let $(X',B')$ be a pair such that $K_{X'}+B'$ is $\Q$-Cartier with log resolution $\phi: (X,B)\rightarrow (X',B')$ where $K_X+B = \phi^*(K_{X'}+B')$.
Assume that $S'$ is the normalisation of a component of ${B'}$ with coefficient $1$, 
and that $S$ is its birational transform on $X$. 
$$
K_X+B=\phi^*(K_{X'}+B'),
$$
let $B_S=(B-S)|_S$. We get 
$$
K_S+B_S= (K_X+B)|_S.
$$
Let $\psi$ be the induced morphism $S\to S'$ and let $B_{S'}=\psi_*B_S$.
Then we get
$$
K_{S'}+B_{S'}= (K_{X'}+B')|_{S'}
$$
which we refer to as  \emph{divisorial adjunction}. 
Note that 
$$
K_{S}+B_{S}=\psi^*(K_{S'}+B_{S'}).
$$
\begin{rem}
	We sometimes write $B_{S'} $ as $\Diff(B'-S')$. Also note that if $C'$ is $\Q$-Cartier that doesn't contain $S'$, then $\Diff(B'+C'-S')=\Diff(B'-S')+C'|_{S'}$.
\end{rem}
\begin{rem}
	We also note that $K_{S'}+B_{S'}$ is determined up to linear equivalence and $B_{S'}$ is determined as a $\Q$-Weil divisor.
\end{rem}
\begin{rem}
	This definition is the same as the definition given in \cite{Kol}, Chapter 4.
\end{rem}
Assume $(X',B')$ is lc. Then the coefficients of $B_{S'}$  belong to $[0,1]$ [\cite{BZh},  Remark 4.8] and we have $(S',B_{S'})$ is also lc.

(3) 
We have inversion of adjunction.

\begin{lem}[inversion of adjunction,\cite{Kol}]\label{l-inv-adjunction}
	Let $(X',B')$ be a $\Q$-factorial pair with $K_{X'}+B'$ is $\Q$-Cartier. 
	Assume $S'$ is a component of ${B'}$ with coefficient $1$ and assume $S'$ is klt Let  
	$$
	K_{S'}+B_{S'}= (K_{X'}+B')|_{S'}
	$$
	be given by adjunction. If $({S'},B_{S'})$ is lc, then $(X',B')$ is lc near $S'$.
\end{lem}

(4) 
The next lemma is similar to [\cite{PSh-II}, Proposition 3.9] and [\cite{BZh}, Proposition 4.9].

\begin{lem}[\text{[\cite{B-Fano}, Lemma 3.3]}]\label{l-div-adj-dcc}
	Let $\mathfrak{R}\subset [0,1]$ be a finite set of rational numbers.
	Then there is a finite set of rational numbers $\mathfrak{S}\subset [0,1]$ 
	depending only on $\mathfrak{R}$ satisfying the following. 
	Assume 
	\begin{itemize}
		
		\item $(X',B')$ is lc of dimension $d$,
		
		\item $S'$ is the normalisation of a component of $\rddown{B'}$,  
		
		\item $B'\in  \Phi(\mathfrak{R})$, and 
		
		\item $(S',B_{S'})$ is the pair determined by  adjunction
		$$
		K_{S'}+B_{S'}= (K_{X'}+B')|_{S'}.
		$$ 
	\end{itemize}
	
	Then $B_{S'}\in  \Phi(\mathfrak{S})$. 
\end{lem}

We will later define similar notion for slc adjunction.
\subsection{Adjunction from Normal pairs to Sdlt Pairs}
Firstly assume $(X,B)$ is log smooth, $B$ is reduced. Since $B$ is a normal crossing divisor, we have $\omega_B$ is a Cartier divisor, hence it corresponds to a Weil divisor $K_B$, whose support doesn't contain any conductor of $B$. Following \cite{Kol} 4.2, we see that we have $\omega_X(B)\cong \omega_B$. \\

\subsubsection{Dlt Adjunction}
Let $(X',B')$ be a pair, and let $S' := \rddown{B'}$. Let $f: (X,B)\rightarrow (X',B')$ be a log resolution that is any exceptional divisor $E$ of $f$ has $a(E,X',B')>-1$ and $K_X+B=f^*(K_X+B)$. Then we have $f$ is isomorphism over generic points of $S'$. Let $S := \rddown{B}$, and let $g: S\rightarrow S'$ be the induced morphism. If we let $B_S := (B-S)|_S$ and $B_{S'}=g_*(B_S)$, we see that we have $K_S+B_S=(K_X+B)|_S$ and we have $$K_{S'}+B_{S'}= (K_{X'}+B')|_{S'}$$ Sometimes, we will write $B_{S'}$ as $\Diff(B'-S')$. Notice again we have $K_S+B_S=g^*(K_{S'}+B_{S'})$. It is clear that $(S',B_{S'})$ is a semi-dlt pair, see \cite{Fuj1}, \cite{Kol} Chapter 5. \\\\
 If we write $S' := \cup_i S_i$, where $S_i$ are the irreducible components of $S'$, then $S_i$ are normal and the normalisation of $S'$ is simply just $S^\nu := \sqcup S_i$. let $\pi:S^\nu\rightarrow S'$ be normalisation and let $K_{S^\nu}+B^\nu=\pi^*(K_{S'}+B_{S'})$ as in 4.1 and let $B_{S_i}=B^\nu|_{S_i}$. Then it is clear $K_{S_i}+B_{S_i}$ is just the divisorial adjunction of $(X,B)$ on the $S_i$ in the sense that $$K_{S_i}+B_{S_i}=(K_X+B)|_{S_i}$$ Hence it is not hard to see that Lemma \ref{l-div-adj-dcc} and \ref{l-inv-adjunction} also hold for this type of adjunction. 
\subsubsection{General Case}
Consider the case where we have a pair $(X',B' := S'+R')$, with $(X',S')$  dlt and $S'$  reduced and $R'$ not containing any irreducible component of $S'$. Then we can define divisorial adjunction similar to above. Letting $R_{S'} := R|_{S'}$, (we note that $R_{S'}$ is indeed a well defined $\Q$-divisor, whose support doesn't contain the components of the conductor divisor on $S'$), and $K_{S'}+\Diff(0) := (K_{X'}+S')|_{S'}$ and letting $B_{S'} := \Diff(0)+R_{S'} $, then we have $$K_{S'}+B_{S'} = (K_{X'}+B')|_{S'}$$ Hence we have the following lemma, which is similar to inversion of adjunction.
\begin{lem}\label{prop-sdlt-inversion}
	Let $(X',B' := S'+R')$ be a $\Q$-factorial pair with $(X',S')$ dlt, $S'$ reduced and $R'$ not  containing any irreducible components of $S'$. Assume $B',R'$ both $\Q$-divisors and write $K_{S'}+B_{R'} := (K_{X'}+B')|_{S'}$. Assume $(S',B_{R'})$ is slc, then $(X',B')$ is lc near a neighbourhood of $S'$.
\end{lem}
\begin{proof}
	Since $(X',S')$ is dlt, then all the irreducible components $S_i$ of $S'$ are normal. Then $\bar{S}:=\sqcup S_i\xrightarrow{\pi} S'$ is its normalisation. Then let $K_{\bar{S}}+\Theta := f^*(K_{S'}+B_{S'})$. Then we see that $(\bar{S},\Theta)$ is lc. Now we can apply [\cite{Kol}, Theorem 4.9], we see that $(X',B')$ is lc near $S'$.
\end{proof} 
\newpage
\section{ Property of Slc and Sdlt Pairs and Adjunction}
Here we will review some fundamental properties of slc and sdlt pairs. Most of this section is based on \cite{Kol} Chapter 5 and \cite{Fuj1}. \\
\subsection{Divisors on Deminormal Schemes}
Starting with standard notation, let $(X,B)$ be an slc pair, with conductor divisor $D$ and normalisation $\pi: (X',B'+D')\rightarrow (X,B)$, where $B'$ is the divisorial part of inverse image of $B$ and $D'$ is the conductor on $X'$. We will firstly introduce the notion of divisors on $X$. There is a closed subset $Z\subset X$ of codimension at least 2 such that $X_0 := X\backslash Z$ has only regular and normal crossing points. We denote $j: X_0 \rightarrow X$ be the inclusion. Let $C$ be an integral Weil divisor on $X$ such that its support doesn't contain any irreducible components of $D$, then $$\mathcal{O}_X(C) := j_*(O_{X_0}(C|_{X_0}))$$ is a divisorial sheaf since $C|_{X_0}$ is Cartier on $X_0$ ([\cite{Kol},5.6.3]). It is also clear that $\mathcal{O}_X(C)^{[n]} $ is the sheaf corresponding to the Weil divisor $nC$ since $X$ is $S2$. Similarly we can define $K_X$ to be the pushforward of the canonical divisor on $X_0$. Now for C as above, we can confuse the notation of $\mathcal{O}_X(mK_X+C)$ with $\omega_X^{[m]}(C)$.\\\\
By [\cite{Kol} 5.7.3], we know if $mB$ is a Weil divisor on $X$, then we have a canonical isomorphism $$(\pi^*\omega_X^{[m]}(mB))^{**} \cong \omega_{X'}(mD'+mB)$$ and when $m(K_X+B)$ is Cartier, this simplifies to $\pi^*\omega_X^{[m]}(mB)\cong \omega_{X'}(mD'+mB)$, hence we can write $K_{X'}+B'+D' = \pi^*(K_X+B)$. This explains the notation we have in the Preliminary. 

\subsection{Conductor and Involution and Adjunction}
Keeping with the notation of an slc pair $(X,B)$ with conductor divisor $D$ and normalisation $(X',B'+D')$. Let $D_n$ be the normalisation of $D'$, then there is a natural Galois involution $\tau : D^n \rightarrow D^n$ induced by separating the nodes on $X$. By divisorial adjunction we can write $K_{D^n}+B_{D^n} := (K_{X'}+B'+D')|_{D^n}$. Rigorously we mean if we let $\nu : D^n \rightarrow X'$ be the natural morphism, and by [\cite{Kol} Chapter 4], we have a natural isomorphism $$\mathfrak{R}^n : \nu^*(\omega_{X'}^{[n]}(nB'+nD') )\cong \omega_{D^n}^{[n]}(nB_{D^n})$$
We also have the following lemma.
\begin{lem}[\cite{Kol}, Prop. 5.38]
	Let $X$ be a quasi-projective deminormal scheme and let $\pi:X'\rightarrow X$ be normalisation. Let $B$ be a $\Q$-divisor on $X$ and define $B',D',D^n,\tau$ as above. Then $(X,B)$ is slc if and only if $(X',B'+D')$ is lc and $\Diff_{D^n}(B')$ is $\tau$-invariant.
\end{lem}
\subsubsection{sdlt case}\label{sec:sdlt-case} If $(X,B)$ is sdlt, then we can use much easier notation. Again let $(X',B'+D')$ be its normalisation and $X' := \sqcup X_i$ and $X_i$ are irreducible components of $X$, Also denote $K_{X_i}+B_i+D_i := (K_{X'}+B'+D')|_{X_i}$. Then we have $$D^n := \displaystyle \sqcup_i(\sqcup_{E \;\text{a component of}\; D_i} E)$$ Also note that since $(X',B'+D')$ is dlt, all components of $D'$ are normal therefore the above adjunction is really just the restriction of divisors (for some properly chosen $K_{X_i}$). In this case the above $\mathfrak{R}^n$ map becomes $$\mathfrak{R}^n: O_{X'}(nK_{X'}+nB'+nD')|_{D^n} \rightarrow \mathcal{O}_{D^n}(nK_{D^n}+nB_{D^n})$$ sending $$(s_i)_i \rightarrow ((s_i|_E)_{E \;\text{a component of}\; D_i})_i$$
Therefore for the dlt case, it is fine to confuse $\mathfrak{R}^n(s)$ with just $s|_{D^n}$.

\subsection{Pre-admissible Sections and Admissible Sections}
Here we quickly review the definition of preadmissible sections and admissible sections for dlt pairs. Most of this sections is from \cite{Fuj1}. We will continue to use the notation that we have used in \ref{sec:sdlt-case}.
\begin{defn}
	Let $(X,B)$ be a sdlt pair with dimension $n$ and we assume $m(K_X+B)$ is Cartier and let $(X',B'+D')$ be its normalisation and $D^n$ be the normalisation of $D'$ as in the above section. \begin{enumerate}
		\item We say $s\in H^0(X,\mathcal{O}_X(m(K_X+B))$ is preadmissible if $s|_{D^n} \in H^0(D^n,\mathcal{O}_{D_n}(m(K_{D^n}+B_{D^n})) := O_{X'}(nK_{X'}+nB'+nD')|_{D^n})$ is admissible. This set is denoted by $PA(X,m(K_X+B))$.
		\item We say $s\in H^0(X,\mathcal{O}_X(m(K_X+B))$ is admissible if $s$ is preadmissible and $g^*(s|_{X_j}) = s|_{X_i}$ for every $g : (X_i,B_i+D_i)\dashrightarrow (X_j,B_j+D_j)$ B-birational map, where $X' := \sqcup X_i$. This set of section is denoted by $A(X,m(K_X+B))$.
	\end{enumerate}
\end{defn}
\begin{rem}$\;$\begin{enumerate}
		\item It is clear that if $s$ is admissible, then $s|_{X_i}$ is invariant under B-birational automorphism for each $(X_i,B_i+D_i)$.
		\item See [\cite{Gongyo-1}, Remark 5.2]. Assume $(X,B)$ is sdlt and $m(K_X+B)$ is Cartier. Let $\pi: (X',B'+D')\rightarrow X$ be its normalisation. Then it is clear that $s\in H^0(X,m(K_X+B))$ is (pre)-admissible if and only if $\pi^* s\in H^0(X',m(K_X'+B'+D'))$ is (pre)-admissible. 
	\end{enumerate}
		\end{rem}
\subsection{Descending Sections to Slc Pairs from Normalisation}
Firstly we start with stating [\cite{Kol} 5.8]. For simplicity we will assume that $m$ is even in the original proposition.
\begin{prop}[\cite{Kol} 5.8]\label{prop-kol-5.8}
	Let $X$ be a demi-normal scheme and $B$ a $\Q$-divisor on X such that support doesn't contain any components of the conductor $D$. Assume $m(K_X+B)$ is Cartier and assume that $m$ is even and let $(X',B'+D')$ be its normalisation.  Then a section $\phi\in \omega_{X'}^{[m]}(mB'+D')$ descends to a section of $\omega_X^{[m]}(mB)$ if and only if $\mathfrak{R}^m\phi$ is $\tau$-invariant where $\tau : D^n\rightarrow D^n$ is the involution and $\mathfrak{R}^m$ is the map defined above.
\end{prop}

We will show the following lemma is the same as [\cite{Fuj1} Lemma 4.2].

\begin{lem} \label{lem-descending-sections}
	Let (X,B) be an slc pair with $m(K_X+B)$ integral. Let $\pi : (X',B'+D')\rightarrow X$ be its normalisation. Let $X' := \sqcup X_i$ and let $K_{X_i}+B_i+D_i := (K_{X'}+B'+D')|_{X_i}$. Let $(Y,B_Y+D_Y)$ be a $\Q$-factorial dlt model of $(X',B'+D')$ in the sense that $Y := \sqcup Y_i$ and let  $K_{Y_i}+B_{Y_i}+D_{Y_i} := (K_Y+B_Y+D_Y)|_{Y_i}$ then $(Y_i,B_{Y_i}+D_{Y_i})$ is a dlt model of $(X_i,B_i+D_i)$. Assume that $m(K_Y+B_Y+D_Y)$ is Cartier.\\\\
	Now let $s\in PA(Y,m(K_Y+B_Y+D_Y)$, then $s$ descends to a section in $H^0(X,m(K_X+B))$.
\end{lem}
\begin{proof}
	It is clear that $m(K_{X'}+B'+D')$ is also Cartier and $s$ descends to a section $s'$ on $H^0(X',m(K_{X'}+B'+D'))$. Notice that $D^n$ the normalisation of $D'$ is nothing but the disjoint union of all the components of $D_Y$. Hence $s'|_{D^n}$ is in particular invariant under the involution $\tau$. 
	\\\\Let $Z$ be a codimensional 2 subset of $X$ such that $X_0 := X\backslash Z$ only has regular or normal crossing points. Then by $m(K_X+B)|_{X_0}$ is Cartier and we can apply \ref{prop-kol-5.8} to get a section $t \in H^0(X_0,m(K_X+B)|_{X_0})$ descending from $s'$. 
	Now notice that $X$ is $S2$, so we get a section $t\in H^0(X,m(K_X+B))$, as required 
\end{proof}

Now we will state the result that we need for later sections. 
\begin{prop}\label{prop-bnd-index-slc}
	Assume $(X,B)$ is an slc pair with normalisation $\pi: (X',B'+D')$ and let $(Y,B_Y,D_Y)$ be the dlt model. Assume that we have $n(K_Y+B_Y+D_Y)\sim 0$ and $0\neq s\in H^0(Y,n(K_Y+B_Y+D_Y)))$ is pre-admissible. Then $n(K_X+B)\sim 0$.
\end{prop}
\begin{proof}
	Firstly note that $n(K_Y+B_Y+D_Y)\sim 0$ implies $n(K_X+B)$ is integral. Hence we can apply \ref{lem-descending-sections} to get a section $0\neq t\in H^0(X,n(K_X+B))$. Also let $s'$ be the corresponding section on $X'$.\\\\
	Again let $Z$ be a codimensional 2 subset of $X$ such that $X_0 := X\backslash Z$ only has regular or normal crossing points. Then $n(K_X+B)|_{X_0}$ is Cartier. let $X'_0 := \pi^{-1}X_0$, we have $\mathcal{O}_{X'_0}(n(K_{X'}+B'+D')|_{X'_0})\cong \mathcal{O}_{X'_0}$ via $s'$, hence we get $\mathcal{O}_{X_0}(n(K_X+B)|_{X_0})\cong \mathcal{O}_{X_0}$ via $t$. Therefore we conclude that $\mathcal{O}_X(n(K_X+B))\cong \mathcal{O}_{X}$ since $X$ is S2. 
\end{proof}

\newpage
\section{Sdlt log Calabi Yau Pair}
In this section, we will review some results from \cite{Fuj1} and \cite{Fujino-Gongyo-1}. We will also show some partial results regarding boundedness of index of sdlt Calabi Yau pairs. Most of the results in this section are included in \cite{Fuj1}, \cite{Gongyo-1} and \cite{Fujino-Gongyo-1}. For reader's convenience, we will include most of the proofs here. Notice that results in \cite{Fuj1} and \cite{Gongyo-1} normally doesn't care about boundedness. We will show that results in \cite{Fuj1} and \cite{Gongyo-1} hold with boundedness provided that some assumptions hold and prove these assumptions hold for curves.
\subsection{Reduced Boundary of Dlt Pairs}
We will show the following basic property of reduced boundary for dlt pairs over a fibration. This is very similar to [\cite{Fuj1}, Proposition 2.1].  For more detailed proof, see [\cite{Fuj1}, Proposition 2.1],[\cite{Gongyo-1}, Claim 5.3]. 

\begin{prop}\label{prop-dlt-boundary}[\cite{Fuj1}, Proposition 2.1],[\cite{Gongyo-1}, Claim 5.3]
	Let $X,B$ be an n-dimensional connected $\Q$-factorial dlt pair. Assume $K_X+B\sim_\Q 0$. Then one of the following holds.
	\begin{enumerate}
		\item  $\rddown{B}$ is connected. 
		\item $\rddown{B}$ has 2 connected components $B_1,B_2$ and there is a rational morphism $(X,B)\dashrightarrow (V,P)$ with general fiber $\mathbb{P}^1$, such that $(V,0)$ is lt and $(V,P)$ is $\Q$-factorial lc of dimension $n-1$. Furthermore, there is horizontal components $S_i$ in $B_i$ such that $(S_i,\Diff(B-S_i))\dashrightarrow (V,P)$ is B-birational.  
	\end{enumerate}
\end{prop}

To show the proposition we first show two easy consequences facts about Mori fiber space.
\begin{lem}\label{lem-dlt-MFS-boundary}
	Let $(X,B)$ be a $\Q$-factorial lc $n$-fold with $n\geq 2$ and $\rddown{B}\neq 0$ and $(X,B-\epsilon\rddown{B})$ is klt for some  small positive rational number $\epsilon$. Let $f : X\rightarrow R$ be a projective surjective morphism with connected fibers such that $K_X+B\sim_Q 0/R$. Assume that there is a $(K_X+B-\epsilon \rddown{B})$ Mori fiber space $g : X\rightarrow V$ over $R$ with $\dim V =n-1$. Let $B_h$ be the horizontal part of $\rddown{B}$. Then either one of the following holds.
	\begin{enumerate}
		\item $B_h = D_1$, which is irreducible and degree $[D_1:V]=2$.
		\item $B_h = D_1$, which is irreducible and degree $[D_1:V]=1$.
		\item $B_h = D_1+D2$, which are irreducible and degree $[D_i:V]=1$.
	\end{enumerate}
    Furthermore, in case (1) and (3), the number of connected components of $\rddown{B}\cap f^{-1}(r)$ is at most 2 and in case (2), $\rddown{B}\cap f^{-1}(r)$ is connected for every $r\in R$. Also, $(V,0)$ is lt and $(V,P)$ is $\Q$-factorial lc $(n-1)$ fold for some P, such that $K_{D_i}+\Diff(B-D_i)=g^*(K_V+P)$ for $i=1,2$. In the case (1), there is a B-birational involution $i: (D_1,\Diff(B-D_1))\dashrightarrow(D_1,\Diff(B-D_1))$ over $V$ such that $i^2=id$. In the case $(3)$, there is a crepant birational involution $j : (D_1,\Diff(B-D_1))\dashrightarrow (D_2,\Diff(B-D_2))$ over $V$.
\end{lem}
\begin{proof}[Proof of Lemma \ref{lem-dlt-MFS-boundary}]
	We firstly note that the general fiber is $\mathbb{P}^1$. Also since $\rddown{B}$ is ample over $R$, we have $B_h\neq 0$. Also since $K_X+B\sim_\Q 0$, restricting to general fiber, we see that $deg(B_H,V)\leq 2$. Hence we see that (1),(2),(3) are the only possibilities. The divisor $P$ can be constructed using generalised adjunction. We can get a $\Q$-divisor $P$ on $V$ such that $K_{D_i}+\Diff(B-D_i)=g^*(K_V+P)$. In particular $(V,P)$ is lc. We note that $V$ is $\Q$-factorial since $X$ is and $g$ is extremal.\\Finally, in case (1) and (3), the part about involution follows from the fact that $deg(B_H,V)=2$. 

\end{proof}
\begin{lem}\label{lem-MFS-boundary-connectedness}
		Let $(X,B)$ be a $\Q$-factorial lc $n$-fold with $n\geq 2$ and $\rddown{B}\neq 0$ and $(X,B-\epsilon\rddown{B})$ is klt for some small positive rational number $\epsilon$. Let $f : X\rightarrow R$ be a projective surjective morphism with connected fibers such that $K_X+B\sim_Q 0/R$. Assume that there is a $(K_X+B-\epsilon \rddown{B})$ Mori fiber space $g : X\rightarrow V$ over $R$. Then either $\rddown{B}\cap f^{-1}(r)$ is connected for every $r\in R$ or $dim V= n-1$ (i.e. we are the in case of \ref{lem-dlt-MFS-boundary}).
\end{lem}
\begin{proof}
	Firstly, lets assume $\rddown{B}\cap f^{-1}(r)$ is not connected for some $r\in R$, this means that $\rddown{B}\cap g^{-1}(v)$ is not connected for some $v\in V$. Noting that $K_X+B\sim_\Q 0/R$ we see that $\rddown{B}$ is g-ample, and hence $B_h$, the horizontal part of $\rddown{B}$, is $g$-ample. Since $\rho(X/V)=1$, we derive that $B_h\cap g^{-1}(v)$ is connected unless the general fiber is $\mathbb{P}^1$. Now it is also clear that since $g$ is extremal, the vertical part of $\rddown{B}$ is the pullback of a $\Q$-Cartier $\Q$-divisor on $V$, hence we get $\rddown{B}\cap g^{-1}(v)$ is connected for each $v\in V$ as claimed. 
\end{proof}

\begin{proof}[Proof of Prop \ref{prop-dlt-boundary}]
	 Now run an MMP on $K_X+B-\epsilon\rddown{B}$  for some $\epsilon>0$ small rational number. Since $K_X+B\sim_\Q 0$, we know that we will end with a Mori Fiber space $g : X'\rightarrow V/R$ with $X\dashrightarrow X'$ a sequence of flips and divisorial contractions. Let $B'$ be the pushforward of $B$ to $X'$. Notice that since $\rddown{B}$ is relatively ample for each divisorial contraction and flip, we know that the number of connected components of $\rddown{B}$ doesn't change during MMP. We can replace $(X,B)$ with $(X',B')$, since all conditions are preserved (because we have $K_X+B\sim_\Q 0$). Notice now $(X',B')$ may not be dlt but it is still lc and $\Q$-factorial. We can finish the proof using Lemma \ref{lem-MFS-boundary-connectedness} and \ref{lem-dlt-MFS-boundary}.
\end{proof}

\subsection{Main Results on Boundedness of Index for Sdlt Log Calabi Yau Pairs}
The goal of this section is to show Theorem \ref{thm-bnd-dlt-index-2}. Firstly, we will note that here a dlt pair will possibly have many connected components. We will make sure to distinguish between these cases. Firstly we need to deal with the klt case. We first use the following proposition to show that admissible sections exist.
\begin{prop}[\cite{Gongyo-1}, Section 5, Theorem C]
 If $(X,B)$ is klt (not necessarily connected) of dimension $d$ and assume $K_X+B\sim_\Q 0$, then for sufficiently large and divisible $N$, there exists a nonzero admissible section in $H^0(X,N(K_X+B))$.
\end{prop}
However to prove boundedness result, we need the following conjectures.
\begin{conj}\label{conj-bounded-klt-admissible-nonconnect}
	If $(X,B)$ is klt (not necessarily connected) of dimension $\leq d$ and assume $n(K_X+B)\sim 0$, then there exists a constant $N(n,d)>0$, such that there is a nonzero admissible section in $A(X,N(K_X+B))$.
\end{conj}
We will prove the above conjecture in this case when $X$ is dimension 1. Firstly, we note that above holds if and only if the following holds. 

\begin{conj}\label{conj-bounded-klt-admissible}
	If $(X,B)$ is connected klt of dimension $\leq d$ and assume $n(K_X+B)\sim 0$, then there exists a constant $N(n,d)>0$, such that there is an admissible section in $A(X,N(K_X+B))$.
\end{conj}

We show that \ref{conj-bounded-klt-admissible} implies $\ref{conj-bounded-klt-admissible-nonconnect}$.
\begin{proof}[Proof of \ref{conj-bounded-klt-admissible} in $\dim$ implies \ref{conj-bounded-klt-admissible-nonconnect} in $\dim$ d:] $\;$\\
	 Let $(X,B)= \sqcup (X_i,B_i)$. Let $s := (\lambda_1 s_1,\lambda_2 s_2,..)\in H^0(X,N(K_X+B))$, where $s_i\in A(X_i,N(K_{X_i}+B_i))$ and $\lambda_i\in \mathbb{C}$. Now let $G := \rho_N(\Bir(X,B))$, which is a finite group by \cite{Fujino-Gongyo-1}. Define $$t := \sum_{\sigma\in G} \sigma(s)$$ Then $t\in A(X,n(K_X+B))$ by construction. Hence it suffices to show we can choose $\lambda_i$ such that $t$ is not zero in all components. To this end, by considering orbits of the action, we can assume $\Bir(X,B)$ acts on $X_{i}$ transitively, i.e. for each $i,j$ there is  $g\in \Bir(X,B)$ mapping $X_{j}$ into $X_{j}$. We notices that $\rho_N(g)$ can be expressed as a matrix such that if entries on diagonal is either 0 or 1. Hence we see that $\sum_{\sigma \in G} \sigma$ is not the zero matrix. Hence there exists some $\lambda_i\in \C$ such that $t$ is not zero on all components. Then since $\Bir(X,B)$ acts transitively and $t$ is $G$-invariant, we see that $t$ is non-zero in all components.
\end{proof}
Now we show that \ref{conj-bounded-klt-admissible} hold for curves.
\begin{proof}[\ref{conj-bounded-klt-admissible} for curves]
	Let $(X,B)$ be a lc curve where $n(K_X+B)\sim 0$. We see that $X$ is either a rational curve or elliptic curve. We can assume that $n$ is even. In either case, we claim that $|G|$, where $G:=\rho_n(\Bir(X,B))$, is bounded depending only on $n$: If $X$ is a rational curve, then we see that $2n\geq |Supp B|\geq 3$, and hence $\Bir(X,B)=Aut(X,B)\leq 6{{2n}\choose{3}}$.  If $X$ is an elliptic curve then $B=0$ and it is well known that $\rho_{12}(Aut(X))$ is trivial.[\cite{kol-old} 12.2.9.1]. In either case, if we let $s\in H^0(X,n(K_X+B))$ be any non-zero section, we see that $0\neq \prod_{\sigma \in G} \sigma(s)\in H^0(X,n|G|(K_X+B)$ is admissible.
\end{proof}

Now we will show 2 statements using induction. Although, we will only apply it for surfaces and curves. The proof are almost taken from \cite{Fuj1}.
\begin{prop} ($A_d$)\label{prop-adimissable-sections-generation}
	Assuming Conjecture \ref{conj-bounded-klt-admissible} in dimension $\leq d-1$. Let $(X,B)$ be a (not necessarily connected) projective dlt pair of dimension $d$, with $m(K_X+B)\sim 0$ and $m$ is even. Also assume that $mN(d-1,m)|n$ where $N(d-1,m)$ as in \ref{conj-bounded-klt-admissible}, then $PA(X, n(K_X+B))$ is non-trivial.
\end{prop}

\begin{prop}
	($B_d$)	Assuming Conjecture \ref{conj-bounded-klt-admissible} in dimension $\leq d$. Let $(X,B)$ be a (not necessarily connected) projective dlt pair of dimension $d$, with $m(K_X+B)\sim 0$ and $m$ is even. Also assume that $mN(d,m)|n$ where $N(d,m)$ as in \ref{conj-bounded-klt-admissible}, then $A(X, n(K_X+B))$ is non-trivial.
\end{prop}

Before we show the above two proposition we will show the proposition below. Firstly we will show an important lemma.

\begin{lem}[\cite{Fuj1} 4.5, \& \cite{Gongyo-1} Claim 5.4]\label{lem-fujino-4.5}
	Assume $(X,B)$ is a projective dlt pair (not necessarily connected) with $n(K_X+B)\sim 0$ and $n$ is even. \\Assume $s\in A(\lf B \rf, n(K_X+B)|_{\lf B \rf})$ is non-zero.\\
	Then there exists a nonzero $t\in PA(X,n(K_X+B))$ such that $t|_{\rddown{B}}=s$.
\end{lem}
\begin{proof}
	This proof follows the same route as [\cite{Fuj1} Proposition 4.5]. Note that the lemma is trivial if $\rddown{B}=0$, hence we assume $\rddown{B}\neq 0$. It is clear by definition it suffice to show there is $t\in H^0(X,n(K_X+B))$ such that $t|_{\rddown{B}}=s$. Therefore, we can assume that $X$ is connected. By Prop \ref{prop-dlt-boundary}, we have either $\lf B \rf$ is connected or has 2 connected components. If $\lf B \rf$ is connected, then $H^0(X,n(K_X+B))\rightarrow H^0(\lf B \rf,\mathcal{O}_{\lf B \rf}(n(K_X+B)|_{\lf B \rf}))$ is injective, hence isomorphism since both are 1 dimensional. In this case, we see that the lemma is clear. \\\\
	Now we assume the $\rddown{B}$ has 2 connected components, $B_1,B_2$. In this case we see that $X$ is generically a $\mathbb{P}^1$ bundle over $(V,P)$. More precisely, there is a sequence of flips and divisorial contraction $\phi : X\rightarrow X'$ and a Mori fiber space $g: (X',B')\rightarrow (V,P)$ such that the general fiber of $g$ is $\mathbb{P}^1$ and $K_X+B= g^* (K_V+P)$.  We also remark that $(\rddown{B'},\Diff(B'-\rddown{B'}))$ is slc by \ref{lem-lt-slc-adjunction}. Also there are 2 connected components of $B'$, $B_1',B_2'$, and each component has an irreducible component $D_i$ such that $g_i:= g|_{D_i}: (D_i,\Diff(B'-D_i))\rightarrow (V,P)$ is B-birational. Now it is easy to see that $$H^0(X,n(K_X+B))\cong H^0(X',n(K_{X'}+B'))$$ Also we have [\cite{Fujino-Gongyo-1}, Remark 2.15], $H^0(\rddown{B},n(K_X+B)|_{\rddown{B}})\cong H^0(\rddown{B'},n(K_{X'}+B')|_{\rddown{B'}})$. Hence it suffices to treat $(X',B')$.  Now let $s\in A(\rddown{B'},n(K_{X'}+B')|_{\rddown{B'}})$, and we write $B_{h}'$ and $B_v'$ be the horizontal and vertical part of $\rddown{B'}$ with respect to V. From Proposition \ref{prop-dlt-boundary}, we see that $s|_{D_i}$ is birational invariant in particular, it descend to a section $t\in H^0(V,n(K_V+P))$. We note that $n(K_{D_i}+\Diff(B'-D_i))\sim 0$ is Cartier, hence we get $n(K_V+P)\sim 0$ and in particular is Cartier. Now since $g$ is contraction and hence we get $H^0(X',n(K_{X'}+B'))\cong H^0(V,n(K_V+P))$, therefore $t$ lifts to a section $w\in H^0(X',n(K_{X'}+B'))$. It suffices to show $w|_{\rddown{B'}}=s$ as remarked before.\\\\
	Firstly, We note that $w|_{D_i}$ and $s|_{D_i}$ are different by at most $(-1)^m$ by \cite{Fuj1} and \cite{Kol} using the theory of $\mathbb{P}^1$ linked lc centres. Hence since we assume $n$ is even, we have the desired claim on $B_{h}'$. Next we check on $B_v'$. It is clear that $B_v' = \sum_i g^*(F_i)$ for some $F_i$ irreducible divisor in $\rddown{P}$. Let $E_i := g^*(F_i)$. We will show $s|_{E_i}=w|_{E_i}$. We let $\Theta_i$ be an irreducible components of $E_i\cap D_1$ that dominant $F_i$ (can always do this since $E_i$ intersects $D_i$ non-trivially). In particular we see that $g|_{\Theta_i}:\Theta_i\rightarrow F_i$ is dominant. Hence we have the following diagram. 
	\[\begin{tikzcd}
		H^0(E_i, n(K_{X'}+B')|_{E_i})\arrow[r,"|_{\Theta_i}"] &H^0(\Theta_i,n(K_{X'}+B')|_{\Theta_i})\\
		H^0(D_i,n(K_V+P)|_{F_i})\arrow[r,"id"] \arrow[u,"\cong"] &	H^0(D_i,n(K_V+P)|_{F_i})\arrow[u,"i"]
		\end{tikzcd}\]
	The right vertical map is injective since $\Theta_i\rightarrow D_i$ is dominant, and left vertical map is isomorphism since $D_i$ is seminormal and $g|_{E_i}$ has connected fibers.[\cite{Fuj1}, Prop 4.5] Since we have $s|_{\Theta_i}=w|_{\Theta_i}$ by the horizontal part argument. Hence we have $s|_{E_i}=w|_{E_i}$, which proves the claim. 
\end{proof}

Now we show the following lemma.
\begin{lem}[\cite{Fuj1}, 4.7]\label{lem-PA=A}
	Let $(X,B)$ be a connected projective dlt pair with $n(K_X+B)\sim 0$ and $n$ is even. Assuming $\rddown{B}\neq 0$, then $PA(X,n(K_X+B))=A(X,n(K_X+B))$. 
\end{lem}
Firstly, we will state a well known lemma about crepant birational maps.
\begin{lem}[\cite{Fuj1}4.7 Claim $A_n,B_n$]\label{lem-gongyo-2.16}
[See Lemma 2.16 in \cite{Fujino-Gongyo-1}] Let $f: (X,B)\dashrightarrow (X',B')$ be a B-birational map between projective dlt pairs. Let $S$ be a lc center of $(X,B)$ such that $K_S+B_S := (K_X+B)|_S$. Let $\alpha: (Y,B_Y)\rightarrow (X,B)$ , $\beta: (Y,B_Y)\rightarrow (X',B')$ be a common log resolution such that $K_Y+B_Y=\alpha^*(K_X+B)=\beta^*(K_{X'}+B')$. Then we can find a lc centre $V$ of $(X,\Delta)$ contained in $S$ with $K_V+B_V := (K_X+B)|_V$, an lc center $T$ of $(Y,B_Y)$ with $K_T+B_T := (K_Y+B_Y)|_T$ and a lc centre $V'$ of $(X',B')$ with $K_{V'}+B_{V'} := (K_V+B_V)|_{V'}$ such that the following hold. \begin{enumerate}
	\item $\alpha|_T: (T,B_T)\rightarrow (V,B_V), \beta|_T: (T,B_T)\rightarrow (V',B_{V'})$ is a B-birational morphism. Hence $\beta|_T\circ \alpha|T^{-1}: (B,B_V)\dashrightarrow (V',B_{V'})$ is $B$-birational.
	\item $H^0(S,m(K_S+B_S))\cong H^0(V,m(K_V+B_V))$ by the natural restriction where $m\in \mathbb{N}^+$ such that $m(K_X+B)$ is Cartier. \qed
\end{enumerate}
\end{lem}
Now we will use the above lemma to prove \ref{lem-PA=A}.
\begin{proof}[Proof of \ref{lem-PA=A}]
	It is clear from definition that $A(X,n(K_X+B))\subset PA(X,n(K_X+B))$. Hence it suffices to show $PA(X,n(K_X+B))\subset A(X,n(K_X+B))$. Let $s\in PA(X,n(K_X+B))$, we need to show for any $g\in \Bir(X,B)$, $g^*(s)=s$. Since $H^0(n(K_X+B))$ is 1 dimensional, it suffices to show $(g^*s)|_{\lf B\rf} = s|_{\lf B\rf} $ (since $H^0(X,m(K_X+B))\rightarrow H^0(\lf B \rf,\mathcal{O}_{\lf B \rf}(m(K_X+B)|_{\lf B \rf}))$ is injective). \\\\
	Let $g\in Bir(X,B)$ and let $\alpha,\beta:(Y,B_Y)\rightarrow (X,B)$ be a Szabo log resolution such that $\alpha := g\circ \beta$, i.e. $\alpha,\beta$ are isomorphisms above the generic points of all lc centre of $(X,B)$. Let $\Theta := B_Y^{=1}$, then by standard theory $\Theta\rightarrow \rddown{B}$ has connected fiber and hence we have $\alpha_*\mathcal{O}_{\Theta}=\beta_* \mathcal{O}_{\Theta} =\mathcal{O}_{\rddown{B}}$. Then $\alpha^*,\beta^*$ induces isomorphism from $$H^0(\rddown{B},\mathcal{O}_{\rddown{B}}(n(K_X+B)|_{\rddown{B}}))\cong H^0(\Theta,\mathcal{O}_{\Theta}(n(K_Y+B_Y)|_{\Theta}))$$ Now let $E$ be an irreducible component of $\Theta$ and let $S$ be its birational transform on $X$, which is an irreducible component of $\rddown{B}$, such that $E$ dominates $S$, Then it suffices to show $(\alpha^*s)|_E = (\beta^* s)|_E$. \\\\Now we apply \ref{lem-gongyo-2.16}, we see that we can find lc centre $V$ contained in $S$ and $T$ a lc centre for $(Y,B_Y)$, such that all the condition are satisfied as in \ref{lem-gongyo-2.16}. Note we can take $V'=V\subset S$. Then we have $\alpha|_{T}^*(s|_{V})=\beta|_{T}^*(s)(s|_{V})\in H^0(T,n(K_T+B_T))$ since $s\in PA(X,n(K_X+B))$. However we have $$H^0(E,n(K_E+B_E))\cong H^0(S,n(K_S+B_S))\cong H^0(V,n(K_V+B_V))\cong H^0(T,n(K_T+B_T))$$ Hence we have $\alpha^*(s)|_E =\beta^*(s)|_E$. Since $E$ is arbitrary, we have $\alpha^*(s|_{\rddown{B}})=\beta^s (s|_{\rddown{B}})$ on $\Theta$. Hence we get $g^*s|_{\rddown{B}}=s|_{\rddown{B}}$, which proves the lemma.
\end{proof}

We are now ready to show the above 2 propositions.
We will first show $B_{d-1}$ implies $A_d$.
\begin{proof}[Proof of $B_{d-1}$ implies $A_d$]
	This is precisely Proposition \ref{lem-fujino-4.5}.
\end{proof}

Finally we show $A_d$ implies $B_d$.
\begin{proof}[Proof of $A_d$ implies $B_d$]
	 We will construct an non-trivial element in $A(X,n(K_X+B))$. Let $G=\rho_n(\Bir(X,B))$ is finite. We can wlog $(X_i,B_i)$ in fact can be put into 2 different class: We say $(X_i,B_i)$ is of type 1, if $\rddown{B_i}\neq 0$, we denote this as these pairs as $(X_{i,1},B_{i,1})$. If $(X_i,B_i)$ is klt, i.e. $\rddown{B_i}=0$, then we say this has type 2 and write $(X_{i,2},B_{i,2})$. Using this notation, we can assume $$(X,B) = (\sqcup_i (X_{i,1},B_{i,1}))\sqcup (\sqcup_i (X_{i,2},B_{i,2}))$$ It is clear that $\Bir(X,B)$ maps type 1 into type 1 and type 2 into type 2. Also $G' := Bir(\sqcup(X_{i,2},B_{i,2}))\subset G$ is also finite. To this end, we write $s=(s_1,s_2,s_3,..,t_1,t_2,...)\in PA(X,n(K_X+B))$, where $s_i\in A(X_{i,1},n(K_{X_{i,1}}+B_{i,1}))$ by \ref{lem-PA=A} and $(t_i)_i\in A(X_{i,2},n(K_{X_{i,2}}+B_{i,2})$ by our assumption on \ref{conj-bounded-klt-admissible-nonconnect}.  We can now think of $G$ acts on $(s_i)$ and $(t_i)$ separately. Now let $s=(s_i,t_i)$ be above denoting an element in $PA(X,n(K_X+B))$.\\\\
	 
	 Step 1: We firstly claim that $(s_i)$ is $G$-invariant. let $\sigma\in G$ be represented by $g\in Bir(X,B)$, the claim is true if $g$ maps $X_{i,1}$ into $X_{i,1}$ for all $i$ since $s_i\in A(X_{i,1},n(K_{X_{i,1}}+B_{i,1}))$.  Therefore, we can assume $g$ maps $X_{i,1}$ to $X_{j,1}$, with $i\neq j$. It suffices to show $g^*(s_j)=s_i$, where we view $g|_{X_{i,1}} : (X_{i,1},B_{i,1})\rightarrow (X_{j,1},B_{j,1})$ a $B$-birational map. Now since we are in type 1, we can assume that $\rddown{B_{i,1}}\neq 0$. Let $S$ be an lc centre of $(X_{i,1},B_{i,1})$, we can apply \ref{lem-gongyo-2.16}, we see that we can find lc centre $V$ of  $(X_{i,1},B_{i,1})$ contained in $S$ and $V'$ of $(X_{j,1},B_{j,1})$, such that $g$ induces a B-Birational map from $g' :(V,B_V)\dashrightarrow (V',B_{V'})$, where $K_V+B_V := (K_{X_{i,1}+B_{i,1}})|_{V}$ and $K_{V'}+B_{V'} := (K_{X_{j,1}+B_{j,1}})|_{V'}$. Also it is clear that $H^0(X_{i,1},n(K_{X_{i,1}+B_{i,1}}))\rightarrow H^0(V,n(K_V+B_V))$ is injective hence an isomorphism. Hence we have the following commutative diagram. 
	\[ \begin{tikzcd}
	 H^0(X_{j,1},n(K_{X_{j,1}+B_{j,1}}))\arrow[r,"\sim","g^*"']\arrow[d,"\sim","|_{V'}"']&H^0(X_{i,1},n(K_{X_{i,1}+B_{i,1}}))\arrow[d,"\sim","|_{V}"']\\
	 H^0(V',n(K_{V'}+B_{V'}))\arrow[r,"\sim","g'^*"']&H^0(V,n(K_{V}+B_{V}))	
	 \end{tikzcd}	\]
	 	Also since $V,V'$ have the same codimension in $X_{i,1}$ and $X_{j,1}$, using the definition that $s_i,s_j$ are pre-admissible, we see that $g'^*(s_j|_{V'})=s_i|_{V}$, which using the above isomorphism, we get $g^*(s_j)=s_i$ as claimed. 
	 \\\\
	 Step 2: Now we deal with the type 2 cases. This is done by our assumption that Conjecture \ref{conj-bounded-klt-admissible-nonconnect} holds in dimension $d$.
\end{proof}
\begin{rem}\label{rem-nonklt-case-bounded}
	We remark that if we only assume $\ref{conj-bounded-klt-admissible}$ in dimension $d-1$, then $B_d$ also hold if no connected component of $(X,B)$ is klt, i.e. if all $(X_i,B_i)$ has $\rddown{B_i}\neq 0$.
\end{rem}
Now we are ready to prove \ref{thm-bnd-dlt-index-2}.

\begin{proof}[Proof of Theorem \ref{thm-bnd-dlt-index-2}]
Let $(X,B)$ be an slc pair such that $K_X+B\sim_\Q 0$, $X$ is dimension 2 and $B\in \Phi(\mathfrak{R})$. Using \ref{ACC-2} and possibly replacing $\mathfrak{R}$, we can assume $B\in \mathfrak{R}$. Let $(X',B') := \sqcup (X_i',B_i') \rightarrow (X,B)$ be its normalisation and let $(Y,\Theta) := (Y_i,\Theta_i)\rightarrow (X',B')$ be a dlt model. Then we have $B' \in \mathfrak{R}$ and hence $\Theta\in \mathfrak{R}$. Using Prop. \ref{prop-can-index}, we can assume that there is a bounded $n$ such that $n(K_Y+\Theta)\sim 0$ and $n(K_X+B)$ is integral. Hence by \ref{prop-adimissable-sections-generation}, possibly replacing $n$ by a bounded multiple, we can find a pre-admissible section $s\in PA(Y,n(K_Y+\Theta))$. By Prop. \ref{prop-bnd-index-slc}, we get $n(K_X+B)\sim 0$
\end{proof}

\newpage
\section{Complements for sdlt pairs}
Now we will show a more important result to the theory of Complements for log canonical Fano pairs. 

\begin{prop}\label{prop-complements-sdlt}
	Let $(X,B)$ be an sdlt pair with $f: (X',B'+D') := \sqcup(X_i,B_i+D_i)\rightarrow X$ be its normalisation where $D'$ is the conductor divisor and $\tau: D^n \rightarrow D^n$ is the involution, where $D^n$ is the  normalisation of $D'$. Let $n$ be an even integer. Assume either $\dim X=1$ or $\dim X=2$. Assume there is a $\Q$-divisor $R'  := \sqcup R_i\geq 0$ with $R_i$ $\Q$-divisor on $X_i$, such that:\\
	
	1. $n(K_{X'}+B'+R'+D')\sim 0$,\\
	
	 2. $(Y',B'+R'+D')$ is lc (hence implies that $R'$ doesn't contain any components of $D'$),\\
	 
	 3. $R'|_{D^n}$ is $\tau$ invariant.\\
	 
	 Then letting $R$ be the pushforward of $R'$ to $X$, we have $m(K_X+B+R)\sim 0$ and $(X,B+R)$ is still slc, where $m = nN(\dim X-1,n)$, and $N(\dim X-1,n)$ is as in \ref{conj-bounded-klt-admissible-nonconnect} depending only on $n$.
\end{prop}
\begin{rem}
\begin{enumerate}
	\item Note that the above proposition essentially gives a essential and necessary condition for constructing complements for sdlt pairs. It says n-complements on each irreducible components give a global n-complement if the divisors can be glued up in a trivial sense. We also note that the  condition $(3)$ is needed since it is  satisfied if $R'$ is the pullback of an n-complements from $X$.
	\item We remark that if \ref{conj-bounded-klt-admissible} holds in all dimension, then the proposition will also hold in all dimensions.
\end{enumerate}
	
\end{rem}
\begin{proof}
	It is clear that we can assume that $X$ is connected. The proof is using results we proved earlier. Here we use Lemma [\cite{Kol}, Prop 5.38]. We see that indeed we have $(X,B+R)$ is slc and $K_{X'}+B'+R'+D' = f^*(K_X+B+R)$. Now using \ref{prop-adimissable-sections-generation} we get we can find a nonzero pre-admissible section for $H^0(X',m(K_{X'}+B'+R'+D'))$. By \ref{prop-bnd-index-slc}, this section will descend to a nonzero section of $H^0(X,m(K_X+B+R))$ hence showing that  $m(K_X+B+R)\sim 0$ as claimed.
\end{proof}
We will now also show a result regarding complements.
\begin{prop}\label{prop-gen-lift-comp}
	Let $\mathfrak{R}\subset[0,1]$ be a finite subset of rationals. Let $(X,B)$ be a $\Q$-factorial dlt pair with $B\in \Phi(\mathfrak{R})$ and $-(K_X+B)$ nef and big. Let $S := \rddown{B}$ and $\Delta := B-S$, and let $K_S+B_S := (K_X+B)|_S$. We have $(S,B_S)$ is an sdlt pair. Let $n$ be a positive integer such that $n\mathfrak{R} \subset \mathbb{N}$. Let $\Delta_S :=B_S-\rddown{B_S}$, then $B_S =\Delta_S$. Suppose further $(K_S+B_S)$ has an n-complement: More precisely, suppose there is $R_S\geq 0$ and $B_S^+ := B_S+R_S \geq 0$ such that \begin{enumerate}
		\item $(S,B_{S}^+)$ is lc ,
		\item $n(K_S+B_{S}^+)\sim 0$.
	\end{enumerate}
Then there is an n-complement $K_X+B^+ := K_X+B+R$ for $K_X+B$ with $R\geq 0$ and $R|_S=R_S$.
\end{prop}
\begin{proof}
	 Let $f: (Y,B_Y)\rightarrow X$ be a Szabo log resolution of $(X,B)$, i.e. $f$ is isomorphism over all generic point of all lc centre of $(X,B)$, with $K_Y+B_Y := f^*(K_X+B)$. let $T := B_Y^{=1}$ and we also use $f: T\rightarrow S$ as the birational contraction induced by $f$. Let $\Delta_Y := B_Y-T$ and let $K_T+\Delta_T := (K_Y+T)|_T+\Delta_Y|_T = (K_Y+B_Y)|_T$. We see that $K_T+\Delta_T=f^*(K_S+B_S)$, and since $\Delta_T<1$, we have $B_S<1$ and hence $B_S=\Delta_S$.\\\\
	Now let $N := -(K_Y+B_Y)$ and define $$L := -nK_Y-nT-\rddown{(n+1)\Delta_Y}=nN+n\Delta_Y-\rddown{(n+1)\Delta_Y}$$ which is an integral divisor hence Cartier. We see that $$L-T =K_Y+<(n+1)\Delta_Y>+(n+1)N.$$
	Hence we have $H^1(Y,L-T)=0$ since $N$ is nef and big, $(Y,<(n+1)\Delta_Y>)$ is klt. Therefore we have $$H^0(Y,L)\twoheadrightarrow H^0(T,L|_T).$$ Now notice that $L|_T = nN|_T+n\Delta_T-\rddown{(n+1)\Delta_T}$. Since $n(K_S+B_S^+)\sim 0$, pulling back to $T$, we get $nN|_T=-n(K_T+B_T)\sim f^*(nR_S)$. \\\\ Hence $L|_T\sim f^*(nR_S) +n\Delta_T-\rddown{(n+1)\Delta_T} := G_T $. It is clear that $G_T$ is integral and $G_T>-1$, hence we get $G_T\geq 0$. By above, there exists $G_Y\geq 0$, an integral divisor with $G_Y|_T=G_T$ and $L\sim G_Y$. Pushing forward to get $-n(K_X+B)+n\Delta-\rddown{(n+1)\Delta}\sim G\geq 0$, where $G=f_*G_T$, hence we get $n(K_X+B+\frac{1}{n}G-\Delta+\frac{1}{n}\rddown{(n+1)\Delta})\sim 0$. Here we remark that since $R_S$ doesn't contain any components of the conductor divisor of $S$ and $\Delta_T$ also doesn't contain any components of conductor divisors, we see that $G_T$ also doesn't contain any components of condutor divisor for $T$. Therefore, $G$ doesn't contain any lc centre of $(Y,T)$.\\\\
	Now let $D$ be a component of $\Delta$ with coefficients $1-\frac{r}{m}$ with $r\in \mathfrak{R}$ and $m\in \mathbb{N}$, then $\mu_D(-n\Delta+\rddown{(n+1)\Delta})=-n+\frac{rn}{m}+\rddown{n+1-\frac{r(n+1)}{m}}$. If $\mu_D(-n\Delta+\rddown{(n+1)\Delta})<0$, then we must have $n-\frac{rn}{m}=a+b$, where $a\in \mathbb{N}$ and $0< b<\frac{r}{m}\leq \frac{1}{m}$. This means that $\frac{rn}{m}+b$ is an integer, but this is not possible since $rn\in \mathbb{N}$ by assumption. Hence we have $-n\Delta+\rddown{(n+1)\Delta}\geq 0$.\\\\
	
	Letting $R := \frac{1}{n}G-\Delta+\frac{1}{n}\rddown{(n+1)\Delta}\geq 0$, we have $n(K_X+B+R)\sim 0$. Letting $B^+ := B+R$, we see that $n(K_X+B^+)\sim 0$. Also by earlier remarks, we see that $R$ doesn't contain any lc centre of $(X,B)$.\\\\
	Now $-nf^*(K_X+B+R)=nN+n\Delta_Y-\rddown{(n+1)\Delta_Y}-G$ since $nN+n\Delta_Y-\rddown{(n+1)\Delta_Y}-G= L-G\sim 0$. We also have $(nN+n\Delta_Y-\rddown{(n+1)\Delta_Y}-G)|_T=L|_T-G_T=-n(K_T+B_T)+n\Delta_T-\rddown{(n+1)\Delta_T}-(f^*(nR_S)+n\Delta_T-\rddown{(n+1)\Delta_T})=-nf^*(K_S+B_S+R_S)$. Hence we have $$(K_X+B+R)|_S=K_S+B_S+R_S$$ Since $(S,B_S+R_S)$ is lc, we have $K_X+B^+$ is lc but not klt near $S$ by Lemma \ref{prop-sdlt-inversion}. Now applying connectedness lemma on $-(K_X+B+(1-\epsilon) R)$ (note this is nef and big) for some small $\epsilon>0$, we see that $K_X+B+R$ is lc globally, which proves the proposition.
\end{proof}
\newpage
\section{Complements for Log Fano Varieties, Dimension 1 or 2}
We start by considering the curve and surface case.

\subsection{The Case for Curves}
Firstly we will consider complements on curves. The following is  more or less an obvious fact.
\begin{lem}\label{lem-curve-comp}
	Let $p\in \mathbb{N}$ and $\mathfrak{R}\subset [0,1]$ be a finite set of rational numbers.
	Then there exists a natural number $n$ 
	depending only on $p,\mathfrak{R}$ satisfying the following.  
	Assume $(X,B+M)$ is a projective pair such that 
	\begin{itemize}
		
		\item $(X,B+M)$ is generalised lc with $X$ a smooth curve,
		
		\item $B\in \Phi(\mathfrak{R})$, and $pM$ is integral,
		
		\item $-(K_X+B+M)$ is nef.
	\end{itemize}
	Then there is an $n$-complement $K_{X}+{B}^++M$ of $K_{X}+{B}+M$ with $B^+\geq B$.
\end{lem}

\begin{proof}
	It is clear that $X$ is either an elliptic curve (which implies $B^+=B=M=0$ and $n=1$), or $X$ is rational curve in which case $X$ is Fano hence we can use \ref{t-bnd-compl-usual}.
\end{proof}

Now we consider the case where $(X,B)$ is an sdlt curve, this means that $X$ itself is a a smooth normal crossing curve. This is an easy consequence of \ref{prop-complements-sdlt}.
\begin{prop}\label{prop-sdlt-curve}
	Let  $\mathfrak{R}\subset [0,1]$ be a finite set of rational numbers.
	Then there exists a natural number $n$ 
	depending only on $\mathfrak{R}$ satisfying the following.  
	Assume $(X,B)$ is a projective pair such that 
	\begin{enumerate}		
		\item $(X,B)$ is sdlt curve,
		\item $B\in \Phi(\mathfrak{R})$, and 
		\item $-(K_X+B)$ is nef.
	\end{enumerate}
	Then there is an $n$-complement $K_{X}+{B}^+$ of $K_{X}+{B}$ with $B^+\geq B$.
\end{prop}
\begin{proof}
	Using \ref{lem-curve-comp}, we see that on each irreducible components there is an $n$ complement for some $n$ depending only on $\mathfrak{R}$. We can assume n is even. Now it is clear that we can choose complements such that they are disjoint from the double point locus of $X$ (this is clear in the elliptic curve case and for the rational curve case, this follows from any 2 points on $\mathbb{P}^1$ are linearly equivalent). Hence the condition in proposition \ref{prop-complements-sdlt} is satisfied. This proves the claim.
\end{proof}
\subsection{Complements for Log Fano Surfaces}
\begin{thm}\label{prop-comp-surface-special}
	Let  $\mathfrak{R}\subset [0,1]$ be a finite set of rational numbers.
	Then there exists a natural number $n$ 
	depending only on $\mathfrak{R}$ satisfying the following.  
	Assume $(X',B')$ is a  pair such that 
	\begin{itemize} 
		\item $X'$ is a projective surface, $(X',B')$ is lc, 
		
		\item $B'\in \Phi(\mathfrak{R})$, and 
		
		\item $-(K_{X'}+B')$ is ample.
	\end{itemize}
	Then there is an $n$-complement $K_{X'}+{B'}^+ $ of $K_{X'}+{B'}$ 
	such that ${B'}^+\ge B'$. 
\end{thm}

We will also show that 

\begin{thm}\label{prop-comp-surface-general}
	Let $\mathfrak{R}\subset [0,1]$ be a finite set of rational numbers.
	Then there exists a natural number $n$ 
	depending only on $\mathfrak{R}$ satisfying the following.  
	Assume $(X',B')$ is a  pair , such that 
	\begin{itemize} 
		\item $X'$ is a projective surface, $(X',B')$ is lc, 
		
		\item $B'\in \Phi(\mathfrak{R})$, and 
		
		\item $-(K_{X'}+B')$ is nef and big.
	\end{itemize}
	Then there is an $n$-complement $K_{X'}+{B'}^+ $ of $K_{X'}+{B'}$ 
	such that ${B'}^+\ge B'$. 
\end{thm}
\begin{proof}[Proof of \ref{prop-comp-surface-general}]
	By taking small $\Q$-factorization, we can assume $(X',B')$ is $\Q$-factorial dlt. Let $S:=\rddown{B'}$ and let $K_S+B_S :=(K_{X'}+B')|_{S}$, we see that $(S,B_S)$ is sdlt and $-(K_{S}+B_S)$ is nef. Also by \ref{l-div-adj-dcc}, we see that there is, $\mathfrak{S}$, a finite subset of $\Q\cap[0,1]$, depending only on $\mathfrak{R}$, such that $B_S\in \Phi(\mathfrak{S})$. Hence we see that $K_S+B_S$ has an $n$-complement $K_S+B_S^+$ with $B_S^+\geq B_S$ for some bounded $n$ depending only on $\mathfrak{R}$. We can also assume $n\mathfrak{R}\in \mathbb{N}$. Now we are done by applying \ref{prop-gen-lift-comp}.
\end{proof}
\begin{proof}[Proof of \ref{prop-comp-surface-special}]
	This follows from \ref{prop-comp-surface-general} and considering a $\Q$-factorial dlt model of $(X',B')$.
\end{proof}
As application, We will also show complements exists for a slightly more generalised setting. Note that we will use results in later section for the following propositions. This is fine since we will not use the below proposition anywhere else.
\begin{prop}\label{prop-comp-surface-fano-type}
	Let  $\mathfrak{R}\subset [0,1]$ be a finite set of rational numbers.
	Then there exists a natural number $n$ 
	depending only on $\mathfrak{R}$ satisfying the following.  
	Assume $(X',B')$ is a  pair , such that 
	\begin{itemize} 
		\item $X'$ is a projective surface of log fano type, i.e. there exists $C'$ such that $(X',C')$ is lc and $-(K_{X'}+C')$ is ample.
		\item $(X',B')$ is lc, 
		
		\item $B'\in \Phi(\mathfrak{R})$, and 
		
		\item $-(K_{X'}+B')$ is nef.
	\end{itemize}
	Then there is an $n$-complement $K_{X'}+{B'}^+ $ of $K_{X'}+{B'}$ 
	such that ${B'}^+\ge B'$. 
\end{prop}
\begin{proof}
	By replacing $C'$ with $(1-\epsilon)C'+B'$ we can assume $C'$ is sufficiently close to $B'$. Then $$-(K_{X'}+B')=K_{X'}+C' -(K_{X'}+C'+K_{X'}+B') $$ Hence we can run MMP on $-(K_{X'}+B')$ and since $-(K_{X'}+B')$ is nef, it is semiample. Therefore, we have three cases.
	\begin{enumerate}
		\item $-(K_{X'}+B')$ is big, then by taking an ample model, we are done by \ref{prop-comp-surface-special}.
		\item $-(K_{X'}+B')$ defines a fibration over a curve, then we are done by \ref{prop-comp-surface-fib}.
		\item $K_{X'}+B'\sim_\Q 0$, then we are done by \ref{prop-can-index}.
	\end{enumerate}
\end{proof}
\newpage
\section{Surface Calabi Yau Type Over Curves}
We are almost ready to prove the general case for 3-folds. First we need to review the theory for surface fibred over curves of Calabi Yau type and prove some boundedness result in this direction. 
\subsection{Adjunction for Fiber Spaces}

To construct complements we sometimes come across lc-trivial fibration $X\to Z$. Using canonical bundle formula  we can 
write $K_X+B$ as the pullback of $K_Z+B_Z+M_Z$ where $B_Z$ and $M_Z$ are the discriminant and 
moduli divisors. In order to apply induction we need to control the coefficients of 
$B_Z$ and $M_Z$ in terms of the coefficients of $B$. We do this in the next proposition.  
The existence of $\mathfrak{S}$ is similar to [\ref{PSh-II}, Lemma 9.3(i)]. 

\begin{prop}\label{l-fib-adj-dcc}
	Let $d\in\N$ and $\mathfrak{R}\subset [0,1]$ be a finite set of rational numbers. 
	Then there exist $q\in \N$ and  a finite set of rational numbers $\mathfrak{S}\subset [0,1]$
	depending only on $d,\mathfrak{R}$ satisfying the following. 
	Assume $(X,B)$ is a pair and $f\colon X\to Z$ a contraction such that 
	\begin{itemize}
		\item $(X,B)$ is projective lc of dimension $d$, and $\dim Z>0$, 
		
		\item $K_{X}+B\sim_\Q 0/Z$ and $B\in \Phi(\mathfrak{R})$,
		
		\item $X$ is of Fano type over some non-empty open subset $U\subseteq Z$, and 
		
		\item the generic point of each non-klt center of $(X,B)$ maps into $U$.\\
	\end{itemize}
	Then we can write 
	$$
	q(K_X+B)\sim qf^*(K_Z+B_Z+M_Z)
	$$
	where $B_Z$ and $M_Z$ are the discriminant and moduli parts of adjunction,  
	$B_Z\in \Phi(\mathfrak{S})$, and for any high resolution $Z'\to Z$ the moduli divisor $qM_{Z'}$ is nef Cartier. 
\end{prop}

Let $S$ be a projective surface with klt singularities and with numerically
trivial canonical class $K_S \equiv 0$. Then $S$ is called a \emph{log Enriques	surface}. The minimal index of complementary $n$ is called the \emph{canonical index} of $S$, i.e. $n = \min\{ q\in \N | qK_S \sim 0 \}$. It is known that $n \le 21$ [\ref{Bl}], [\ref{Zhang}]. A similar result also holds for surfaces of Calabi-Yau type. The argument below is the same as [\ref{PSh-II}, Corollary 1.11]. We put it here for the reader's convenience.

\begin{prop}\label{prop-can-index}
	Let $\mathfrak{F}\subset [0,1]$ be a DCC set of rational numbers. Then there exists $n\in \N$ depending only on $\mathfrak{F}$ satisfying the following. Assume $(S,B)$ is a pair such that
	\begin{itemize}
		\item $S$ is a projective surface, $(S,B)$ is lc, and
		
		\item $K_{S}+B\sim_\Q 0$ and, $B\in \mathfrak{F}$.\\
	\end{itemize}
	Then $n(K_S+B)\sim 0$.
\end{prop}
\begin{proof}
	Replacing $(S,B)$ we can assume it is $\Q$-factorial dlt. If $K_S \sim_\Q 0$ (which implies $B=0$), then $S$ is a log Enriques surface and such $n$ exists. We therefore assume $K_S \nsim_\Q 0$. Thanks to global ACC Theorem \ref{ACC-2} we can assume $\mathfrak{F}$ is a finite set. Run an MMP on $K_S$ which terminates at a Mori fiber space $g:S' \rightarrow Z$ as $K_S$ is not pseudo-effective. If $Z$ is a point, then $S'$ is a Fano variety which in turn implies that there is an $n$-complement $K_{S'}+B'^+$ of $K_{S'}+B'$ such that $B'^+ \ge B'$. Since $K_{S'}+B' \sim_\Q 0$, we deduce $B'^+=B'$, hence $n(K_S+B)$ is Cartier. Now we assume $Z$ is a curve. Replacing $(S,B)$ with $(S',B')$ and applying Proposition \ref{l-fib-adj-dcc}, we get a generalized polarized pair $(Z,B_Z+M_Z)$ such that $n(K_S+B) \sim n g^* (K_Z+B_Z+M_Z)$, $B_Z \in \Phi(\mathfrak{S})$ for some finite set $\mathfrak{S}$ depending only on $\mathfrak{F}$, and $nM_Z$ is Cartier. Note that either $Z$ is a rational curve or an elliptic curve. In both cases $M_Z$ is effective semi-ample, by Lemma \ref{lem-curve-comp}, which in turn implies that $n(K_Z+B_Z+M_Z) \sim 0$ after replacing $n$ with some multiple. Hence the conclusion follows immediately. 	
\end{proof}
We first state a theorem by Shokurov.
\begin{prop}[\cite{PSh-II}, Theorem 8.1] 8.1]\label{prop-adjunction-shokurov}
	Let $\mathfrak{R}\subset [0,1]$ be a finite set of rational numbers. 
	Then there exists $q\in \N$ depending only on $\mathfrak{R}$ satisfying the following. 
	Assume $(S,B)$ is a pair and $f\colon S\to C$ a contraction such that 
	\begin{itemize}
		\item $S$ is a projective surface, $(S,B)$ is lc, $C$ is a curve, and
		
		\item $K_{S}+B\sim_\Q 0/C$ and $B_h\in \Phi(\mathfrak{R})$ where $B_h$ is the horizontal part of $B$ over $C$.\\
	\end{itemize}
	Then we can write 
	$$
	q(K_S+B)\sim qf^*(K_C+B_C+M_C)
	$$
	where $B_C$ and $M_C$ are the discriminant and moduli parts of adjunction,  and the moduli divisor $qM_C\geq 0$ is effective semi-ample  and Cartier. 
	
\end{prop}

We will use it to prove theory on relative complements as following.

\begin{lem}\label{lem-rel-complement}
	Let $\mathfrak{R}\subset [0,1]$ be a finite set of rational numbers. 
	Then there exists a natural $n\in \N$ 
	depending only on $\mathfrak{R}$ satisfying the following. 
	Assume $(S,B)$ is a pair and $f\colon S\to C$ a contraction such that 
	\begin{itemize}
		\item $S$ is a surface, $(S,B)$ is lc, and $C$ is a rational curve, and
		
		\item $K_{S}+B\sim_\Q 0/C$ and $B\in \Phi(\mathfrak{R})$.\\
	\end{itemize}
	Then for any point $z \in C$, there is an $n$-complement $K_S+B^+$ of $K_S+B$ over $z$ such that $B^+ \ge B$.
\end{lem}
\begin{proof}
	Firstly, $C$ is normal hence it is smooth. let $z\in C$ and let $D :=f^*z$ be a Cartier divisor. Write $$q(K_S+B)\sim qf^*(K_C+B_C+M_C)$$ as in \ref{prop-adjunction-shokurov}. Note that $qK_C,qM_C$ are both Cartier hence integral. Now let $t := lct(D,S,B)$ be the log canonical threshold of $D$ with respect to $K_S+B$, we see that $\mu_z(B_C) = 1-t$ from definition of canonical bundle formula. This is because that $t$ is also the log canonical threshold of $D$ w.r.t to $K_S+B$ over $z$ (as $z$ is Cartier and $D$ is supported in the fiber over $z$). Hence if we let $B^+ := B+tD$, we see that $(S,B^+)$ is lc by definition and $B^+$ has the same horizontal components as $B$. Hence we get $$q(K_S+B^+)\sim qf^*(K_C+B_C^+ + M_C)$$, where $B_C^+ := B_C+tz$. Hence we get $B_C^+$ is Cartier near $z$ as $\mu_z(B_C^+)=1$. Therefore, we get $q(K_C+B_C+M_C)\sim 0$ in an open neighbourhood of $z$, hence we get $q(K_X+B^+)\sim 0$ over an open neighbourhood of $z$. Now $q$ is the bounded $n$ that we are looking at. The last claim is clear.
\end{proof}
We are now ready to show the following proposition that is similar to \ref{l-fib-adj-dcc}.

\begin{prop}\label{prop-adjunction-hyperstd}
	Let $\mathfrak{R}\subset [0,1]$ be a finite set of rational numbers. 
	Then there exist $q\in \N$ and a finite set of rational numbers $\mathfrak{S}\subset [0,1]$
	depending only on $\mathfrak{R}$ satisfying the following. 
	Assume $(S,B)$ is a pair and $f\colon S\to C$ a contraction such that 
	\begin{itemize}
		\item $S$ is a projective surface, $(S,B)$ is lc, $C$ is a curve, and
		
		\item $K_{S}+B\sim_\Q 0/C$ and $B\in \Phi(\mathfrak{R})$.\\
	\end{itemize}
	Then we can write 
	$$
	q(K_S+B)\sim qf^*(K_C+B_C+M_C)
	$$
	where $B_C$ and $M_C$ are the discriminant and moduli parts of adjunction,  
	$B_C\in \Phi(\mathfrak{S})$, and the moduli divisor $qM_C\geq 0$ is effective semiample and Cartier. 
	
\end{prop}
\begin{proof}
Most parts of the claim are the same as \ref{prop-adjunction-shokurov} except the coefficients of $B_C$. We will show this using $n$-complements. The question is local on $C$. Pick any $z\in C$. Let $B^+\geq B$ be an n-complement $K_X+B$ over $z$ as in \ref{lem-rel-complement}. We see that 
$$q(K_S+B)\sim qf^*(K_C+B_C+M_C).$$ Now we have $\mu_z(B_C)=1-t$, where $B^+ =B +tf^* z$. Let $F$ be an component of $f^*z$ and let $l := \mu_F(f^* z)\in \mathbb{N}$ and let $b := \mu_F(B)$. We have $b=1-\frac{r}{m}$ for some $r\in \mathfrak{R}$ and $m\in \mathbb{N}$. Hence we have $\mu_F(B^+)= 1-\frac{r}{m}+tl$. Also we have $n\mu_F(B^+)\in\mathbb{N}$. Therefore we get $$n\geq n(1-\frac{r}{m}+tl)\in \mathbb{N}.$$
 If $t=0$, then we have nothing to prove. Hence we can assume $t>0$, hence letting $ a := n(1-\frac{r}{m}+tl)$, we have $$t = \frac{\frac{a}{n}-1+\frac{r}{m}}{l}$$ Now if $a=n$, then it is clear that $1-t=1-\frac{r}{ml}\in \Phi(\mathfrak{R})$. If $\frac{a}{n}<1$, we have $\frac{r}{m}>1-\frac{a}{n}\geq \frac{1}{n}$ as $t>0$. Now since $r\leq 1$, we have $\frac{1}{m}>\frac{1}{n}$. Hence $m<n$, and hence there are only finitely many choice for $\frac{r}{m}$, and hence only finitely many choice for $\frac{a}{n}-1+\frac{r}{m}$. Denote this set union $\mathfrak{R}$ to be $\mathfrak{S}$, we see that $\mu_z(B_C)=1-t\in \Phi(\mathfrak{S})$, which proves the claim.
\end{proof}
We end the section by a proposition about global complements for surface fibred over rational curve.

\begin{prop}\label{prop-comp-surface-fib}
	Let $\mathfrak{R}\subset [0,1]$ be a finite set of rational numbers.
	Then there exists a natural number $q$ 
	depending only on $\mathfrak{R}$ satisfying the following.  
	Assume $(X,B)$ is a  pair , such that 
	\begin{itemize} 
		\item $X$ is a projective surface, $(X,B)$ is lc, 
		
		\item $B\in \Phi(\mathfrak{R})$, and 
		
		\item There is a contraction $f:X\rightarrow C$ such that $K_{X}+B\sim_\Q0/C$ and $C$ is rational curve.
		\item $-(K_X+B)$ is nef.
	\end{itemize}
    We write $q(K_X+B)\sim qf^*(K_C+B_C+M_C)$, where $B_C\in\Phi(\mathfrak{S})$, $qM_C$ is effective semiample Cartier divisor and $q,\mathfrak{S}$ as in \ref{prop-adjunction-hyperstd} depending only on $\mathfrak{R}$. Then any $pq$-complement $K_C+B_C^++M_C$ of $K_C+B_C+M_C$  with $B_C^+\geq B_C$ lifts to an $pq$-complement $K_X+B^+$ of $K_X+B$ with $B^+ := B+f^*(B_C^+-B_C)\geq B$. In particular, $K_X+B$ has a $n$-complement for some $n$ depending only on $\mathfrak{R}$.
\end{prop}
\begin{proof}
	By \ref{lem-curve-comp}, $K_C+B_C+M_C$ has $p$-complement for some $p$ depending only  $\mathfrak{R}$. Let $B_C^+ := B_C+D_C$, where $D_C\geq 0$, be such a $p$-complement. Then letting $n=pq$, we have $$n(K_X+B+f^*(D_C))\sim nf^*(K_C+B_C^++M_C)\sim 0.$$ Hence it suffices to show $(X,B+D)$ is lc where $D:= f^*D_C$. This follows from the fact that $(C,B_C^+ + M_C)$ is generalised lc: Indeed, consider a log resolution of $(X,B+D)$, $g:(Y,B_Y+D_Y)\rightarrow X$, where $K_Y+B_Y=g^*(K_X+B)$ and $D_Y := g^* D$. Now if $(X,B+D)$ is not lc, then there exist an irreducible component E such that $\mu_E(B_Y+D_Y)>1$, hence $\mu_E(D_Y)>0$, which means $E$ is vertical over $C$. But by property and definition of the canonical bundle formula, we see that $K_X+B+D \sim_\Q f^*(K_C+B_C^+ +M_C)$, hence if $E$ is mapped to $z\in C$, then $\mu_z(B_C^+)>1$, which is a contradiction.
\end{proof}

Now we need to strengthen the above proposition in order to consider gluing. 
\begin{prop}\label{prop-comp-surface-fib-glue}
	Take the setting and notation as in \ref{prop-comp-surface-fib}. Assume further that $(X,B)$ is dlt.  We assume $q|n$ and $n $ is even. Assume $B_h $, the horizontal part of $\rddown{B}$ over $C$ is non-trivial. Then either $B_h$ is irreducible and $\deg(B_h:C)=1$ or $2$ or $B_h=B_1+B_2$, where $B_i$ irreducible and $deg(B_1:C)=1$. \begin{enumerate}
		\item If $B_h:= B_1$ irreducible and $\deg(B_1:C)=1$.  Then any "general" $n$-complements $K_{B_1}+\Diff(B-B_1)+R_{B_1}$ of $K_{B_1}+\Diff(B-B_1)$ with $R_{B_1}\geq 0$ lifting to an n-complement of $K_X+B+R$ such that $R|_{B_1}=R_{B_1}$.
		\item If $B_h:= B_1+B_2$, $B_i$ irreducible and $\deg(B_i:C)=1$.  Then any "general" $n$-complements $K_{B_1}+\Diff(B-B_1)+R_{B_1}$ of $K_{B_1}+\Diff(B-B_1)$ with $R_{B_1}\geq 0$ lifts to an n-complement of $K_X+B+R$ such that $R|_{B_1}=R_{B_1}$.
		\item If $B_h:= B_1$, $B_1$ irreducible and $\deg(B_1:C)=2$, let $\sigma\in Gal(B_1/C)$ represent the involution. Then any "general" $n$-complements $K_{B_1}+\Diff(B-B_1)+R_{B_1}$ of $K_{B_1}+\Diff(B-B_1)$, with $R_{B_1}\geq 0$  and $R_{B_1}$ is $\sigma$-invariant, lifts to an n-complement of $K_X+B+R$ such that $R|_{B_1}=R_{B_1}$.
	\end{enumerate}
 Here general means that if we write $R_{B_1}=f|_{B_1}^*(R_C)$, then $Supp(R_C)$ is disjoint from $Supp(B_C+M_C)$. Note that such $n$-complements are indeed general if $B_1$ and $C$ are rational curves.
\end{prop}
\begin{proof}
	We firstly note that since $B_i$ are irreducible lc centre of $(X,B)$, we see that $B_i$ are all smooth curves. Restricting to the general fiber, we see that $B_h$ must be one of the cases as we listed. Now we deal with lifting complements in each case respectively.
	\begin{enumerate}
		\item Here we see that $B_1\cong C$ via $f$. Hence we have $q(K_{B_1}+\Diff(B-B_1))= q(K_C+B_C+M_C)$. Now given such $R_{B_1}$, we can first push it down to $C$ via $f|_{B_1}$, call it $R_C$, and then pullback to $X$. The generality will make sure $(K_C+B_C+R_C+M_C)$ is generalised lc. Then the remaining claim is obvious.
		\item This is precisely the same as the previous case.
		\item If $R_{B_1}$ is $\sigma$-invariant, we see that $R_{B_1}=f|_{B_1}^* (R_C)$ for some $R_C\geq 0$. Let $\sigma: B_1\rightarrow B_1$ also denote the isomorphism such that the following diagram commutes, (in general it should be birational map, but since $B_1$ is a smooth curve, it is isomorphism) \\
\[\begin{tikzcd}[column sep=tiny]
  (B_1,\Diff(B-B_1))\arrow[rr,"\sigma"]\arrow[dr,"f|_{B_1}"']& &  (B_1,\Diff(B-B_1)) \arrow[dl,"f|_{B_1}"]
\\
   & (C,B_C+M_C) &
\end{tikzcd}\]
and since we have $q(K_{B_1}+\Diff(B-B_1))= qf|_{B_1}^*(K_C+B_C+M_C)$, then $\sigma^*(K_{B_1}+\Diff(B-B_1))=K_{B_1}+\Diff(B-B_1)$. Also note that functions on $B_1$ descend to functions on $C$ if and only if it is $\sigma$-invariant. Now since $R_{B_1}$ is also $\sigma$-invariant, we see that if $n(K_{B_1}+\Diff(B-B_1)+R_{B_1})\sim 0$ via $F\in \C(B_1)$, then $F\in \C(C)$ and hence $n(K_C+B_C+R_C+M_C)\sim 0$ as required. The rest is same as (1).
	\end{enumerate}
\end{proof}
  
Now we are ready to show complements for 3-folds.
We first need some trivial facts about curves that follow from Riemann Hurwitz formula.
\begin{lem}\label{lem-curve-reimann-hurwitz}
	Let $\mathfrak{R}\subset[0,1]$ be a finite subset of rationals. Let $f: C\rightarrow T$ be a finite morphism between smooth rational curves. Assume $f$ is Galois, i.e. $\mathbb{C}(T)\subset \C(C)$ is a Galois extension. Let $B_C\geq 0$ be a $\Q$-divisor on C such that $-(K_C+B_C)$ is ample and assume $B_C\in \Phi(\mathfrak{R})$. Also assume $K_C+B_C$ is $Gal(C/T)$ invariant. Then there exists $B_T\geq 0$ such that $B_T\in \Phi(\mathfrak{R})$ and $K_C+B_C= f^*(K_T+B_T)$. In particular, there is a $n$, depending only on $\mathfrak{R}$ such that there is an $n$-complement $K_C+B_C+R_C$ for $K_C+B_C$ with $R_C\geq 0$, $R_C$ is in general position and $R_C$ is $Gal(C/T)$ invariant.
\end{lem}
\begin{proof}
	We apply the Reimann-Hurwitz formula. We have $K_C= f^*(K_T) + \displaystyle \sum_{Q\in C} (e_Q-1) Q$ where $e_Q$ is the ramification index at $Q$. Now since $f$ is Galois, we see that $e_Q=e_{Q'}$ if $f(Q)=f(Q')$. Hence we can define $e_P := e_{Q\in f^{-1}P}$ for $P \in T$, which is well defined. It is clear that we have $\displaystyle f^*P = e_P\sum_{Q: f(Q)=P} Q$. Furthermore, the above formula becomes $$K_C =f^*(K_T+\sum_{P\in T}(1-\frac{1}{e_P})P)$$ Now since $B_C$ is $Gal(C/T)$ invariant, we can write $\displaystyle B_C = \sum_{P\in T}a_P (\sum_{Q: f(Q)=P} Q)$, where $a_P\in \Phi(\mathfrak{R})$. Hence we have $B_C = f^*(\displaystyle \sum_{P\in T} \frac{a_P}{e_P}P )$. Hence we have $$K_C+B_C= f^*(K_T+\sum_{P\in T}(1-\frac{1-a_P}{e_P})P)=: f^*(K_T+B_T)$$ Now if $a_P = 1-\frac{r}{m}$ for some $r\in \mathfrak{R}$ and $m\in \N$, then we have $\mu_P (B_T)= 1-\frac{r}{me_P}\in \Phi(\mathfrak{R})$. The last part of the claim is clear by taking an n-complement on $C$ to be the pullback of an $n$-complement on $T$.
\end{proof}

\newpage

\section{Complements for Log Fano Threefold}
Finally, we are ready to prove the main theorem of this paper. 
\begin{thm}\label{thm-comp-3fold-general}
	Let  $\mathfrak{R}\subset [0,1]$ be a finite set of rational numbers.
	Then there exists a natural number $n$ 
	depending only on $\mathfrak{R}$ satisfying the following.  
	Assume $(X',B')$ is a  pair such that 
	\begin{itemize} 
		\item $X'$ is a projective 3-fold, $(X',B')$ is lc, 
		
		\item $B'\in \Phi(\mathfrak{R})$, and 
		
		\item $-(K_{X'}+B')$ is ample.
	\end{itemize}
	Then there is an $n$-complement $K_{X'}+{B'}^+ $ of $K_{X'}+{B'}$ 
	such that ${B'}^+\ge B'$. 
\end{thm}

We need to first state a result in \cite{Kol}.
\begin{prop}[\cite{Kol}, Theorem 4.45]
	 Let $f: (X,B)\rightarrow (X',B')$ be a $\Q$-factorial dlt model, with $(X',B')$ lc. Let $Z$ be an lc centre of $(X',B')$, and let $W$ be a minimal lc centre of $(X,B)$ that dominates $Z$. Let $K_W+B_W := (K_X+B)|_W$, and let $W\rightarrow Z_s\rightarrow Z$ be the stein factorization. Then isomorphism class of $Z_s$ are independent of $W$ and $X$ and B-birational class of $(W,B_W)$ are also independent of $W$ and choice of $X$. Also all such $(W,B_W)$ are $\mathbb{P}^1$-linked. Furthermore, $Z_s\rightarrow Z$ is Galois and $\Bir(W,B_W)\rightarrow Gal(Z_s/Z)$ is surjective.
\end{prop}
Now we are ready to show the main theorem of the paper.
\begin{proof}[Proof of \ref{thm-comp-3fold-general}]
	We can assume that $(X,B)$ is not klt. Let $f: (X,B)\rightarrow (X',B')$ be an $\Q$-factorial dlt model of $(X,B)$ with $K_{X}+B=f^*(K_{X'}+B')$. Let $S := \rddown{B}$ and write $K_S+B_S := (K_X+B)|_S$. Then by Prop. \ref{prop-gen-lift-comp}, it suffices to show $K_S+B_S$ has an $n$-complement with $n$, depending only on $\mathfrak{R}$. First we will show there are complements on each irreducible components of $S$ and then we will use \ref{prop-complements-sdlt} to show they glue to give an complement for $K_S+B_S$. We note that $B_S\in \Phi(\mathfrak{S})$ with $S$ a finite set of rationals depending only on $\mathfrak{R}$. We note that $S$ is connected by connectedness theorem since $-(K_X+B)$ is nef and big.\\\\
	Step 0: We first settle the case when $S$ is irreducible. In this case, let $T'$ be image of $S$ on $X'$ and let $ S\xrightarrow{g} T\rightarrow T'$ be the stein factorization. Then we see that either $T$ has dimension 0, in which case we have $K_S+B_S\sim_\Q 0$ or T is a curve or a surface. We can apply Proposition \ref{prop-can-index} or Proposition \ref{prop-comp-surface-fib} or \ref{prop-comp-surface-special}, to show there is an n-complement with $n$ depending only on $\mathfrak{R}$ for $K_S+B_S$. Now we are done by applying Prop. \ref{prop-gen-lift-comp}. Hence from now on, we will assume $S$ has multiple irreducible components.\\\\
	Step 1: We first consider complements on curves. Let $T$ be an irreducible one dimensional lc centre on $K_X+B$. Write $K_T+B_T := (K_X+B)|_T$ (note this is well defined up to sign by \cite{Kol} and we will always assume $n$ to be even). We note that coefficients of $B_T$ lies in $\Phi(\mathfrak{F})$ with $\mathfrak{F}$ depending only on $\mathfrak{R}$. Then either $T$ is contracted by $f$ or image of $T$ is a curve. \begin{enumerate}
		\item If $T$ is contracted, then $K_T+B_T\sim_\Q 0$ and hence $n(K_T+B_T)\sim 0$ for some $n$ depending only on $\mathfrak{R}$. In this case, we let $R_T=0$ is an $n$-complement for $(T,B_T)$.
		\item If $T$ is not contracted by $f$, let $T'$ be its image on $X$, then we have $K_T+B_T\sim_\Q f|_T^*(-A)$ for some $A$ ample on $T'$. Hence we see that $T,T'$ are rational curves. since $T$ is the minimal lc centre dominating $T'$, we see that $T\rightarrow T'$ is Galois. Furthermore, by \cite{Kol}, If $(\hat{T},B_{\hat{T}})$ is another one dimensional lc centre on $(X,B)$ dominating $T'$, then $(T,B_T)$ are naturally $B$-birational to $(\hat{T},B_{\hat{T}})$ in the sense that we have a commutative diagram \[\begin{tikzcd}
		(T,B_T)\arrow[rr,"\sigma"]\arrow[dr,"f|_T"']&&(\hat{T},B_{\hat{T}})\arrow[dl,"f|_{\hat{T}}"]\\
		&T'&	
		\end{tikzcd}\] We note that $\sigma$ maybe not unique.	
			\\ In this case, we have $-(K_T+B_T)$ is ample and we can find an bounded $n$-complement $K_T+B_T+R_T$ with $R_T\geq 0$ by \ref{lem-curve-comp}. Furthermore, if we let $\bar{T}$ be the normalisation of $T'$, then $T\rightarrow \bar{T}$ is also Galois. Also note that $Gal(T/\bar{T})=Gal(T/T')$. Now since $Bir(T,B_T)\rightarrow Gal(T/T')$ is surjective, we see that $K_T+B_T$ is $Gal(T/\bar{T})$ invariant. Now we can apply Lemma \ref{lem-curve-reimann-hurwitz} and we can assume that $R_T$ is $Gal(T/\bar{T})$ invariant. Now using the isomorphism between $\hat{T}$ and $T$, we can get a compatible set $n$-complement on all $\hat{T}$ mapping to $T'$. Note that although $\sigma: (T,B_T)\rightarrow (\hat{T},B_{\hat{T}})$ is not unique, $\sigma(R_T)$ is well defined since $R_T$ is $Gal(T/\bar{T})$ invariant. Also by generality, we can assume $R_T$ doesn't contain any other lc centre on $(X,B)$.
	\end{enumerate}
 Hence now for each $T$, dimension 1 irreducible lc centre of $(X,B)$, we have contructed an $n$-complement $K_T+B_T+R_T$ with $R_T\geq 0$ and $R_T$ are disjoint from any other lc centre on $(X,B)$.\\\\
 Step 2: Now let $S := \cup S_i$, where $S_i$ are the irreducible components of $S$. Then either $S_i$ are mapped to points, curves or surfaces on $X'$. We distinguish the 3 cases. Let $W$ be a general $S_i$.
 \begin{enumerate}
 	\item If $W$ is mapped to a point on $X'$, then $K_{W}+B_W := (K_X+B)|_W \sim_\Q 0$, and the coefficients of $B_W$ are $\Phi(\mathfrak{S})$ for some finite set of rationals $\mathfrak{S}$ depending only on $\mathfrak{R}$. Hence by Prop. \ref{prop-can-index}, there is $n$, depending only on $\mathfrak{R}$, such that $n(K_W+B_W)\sim 0$, in this case, the n-complement $R_W=0$.
 	\item If $W$ is mapped to a surface on $X'$, then $-(K_W+B_W)$ is nef and big. Let $V := \rddown{B_W}$ and $K_V+B_V := (K_W+B_W)|_V$, we see that $V$ is an sdlt curve and $-(K_V+B_V)$ is nef. In particular, for each irreducible component of $V$, we have already created an $n$-complement with $n$ depending only $\mathfrak{R}$ such that they are disjoint from non-normal locus of $V$. Hence by Proposition \ref{prop-complements-sdlt}, we have already found an $n$-complements for $K_V+B_V$ in the form of $K_V+B_V+R_V$, with $R_V\geq 0$. Therefore, by Proposition \ref{prop-gen-lift-comp}, we can lift these complements to a $n$-complement $K_W+B_W+R_W$ for $K_W+B_W$ such that $R_W|_W := R_V$, i.e. for each irreducible components $T$ in $V$, we have $R_W|_T= R_T$ defined above. 
 	\item The last case is that $W$ is mapped to a curve $T'$. Let $f: W\xrightarrow{g} C\rightarrow T$ be the stein factorization. By Prop. \ref{prop-adjunction-hyperstd}, we can find a $q$ depending only $\mathfrak{R}$, such that $$q(K_W+B_W)\sim q(K_C+B_C+M_C)$$ 
 	We note that $-(K_C+B_C+M_C)$ is $\Q$-linearly equivalent to the pullback of an ample divisor on $T'$, hence we see that $C$ is a smooth rational curve. Now We split into further cases depending on $B_h$, the horizontal over $C$ part of $\rddown{B_W}$.
 	\begin{enumerate}
 		\item Case 1: $B_h =0$, then by  \ref{lem-curve-comp},we can simply choose any $n$-complement $K_C+B_C+R_C+M_C$ for $K_C+B_C+M_C$ and using \ref{prop-comp-surface-fib}, lift to an $n$-complement $K_W+B_W+R_W$ for $K_W+R_W$ with $R_C\geq 0$ and $R_W := g^*(R_C)$. Note that in this case, for any $D$, an irreducible component of $\rddown{B_W}$, we have $R_W|_D =0$ since if $D$ is mapped to $z\in C$, then $\mu_z(B_C)=1$ and hence $\mu_z(R_C)=0$.
 		\item Case 2: $B_h\neq 0$. Then we are in the case of Prop.\ref{prop-comp-surface-fib-glue}. In particular we have already found $n$-complements  on $(B_h,\Diff(B_W-B_h))$ satisfying the criteria for \ref{prop-comp-surface-fib-glue}: Indeed, the there are two non-trivial cases as following. \begin{enumerate}
 			\item $B_h$ is irreducible and $deg(B_h,C)=2$, but in this case, we have $Gal(B_h/C)\subset Gal(B_h/T')$. Hence we complements we have constructed for $K_{B_h}+\Diff(B_W-B_h)$ satisfies the assumption in \ref{prop-comp-surface-fib-glue}. Hence we see that there is a $n$ depending only on $\mathfrak{R}$ such that there is a $n$-complement $K_W+B_W+R_W$ for $K_W+B_W$, with $R_W\geq 0$ lifting $n$-complement $K_{B_h}+\Diff(B_W-B_h)+R_{B_h}$ for $K_{B_h}+\Diff(B_W-B_h)$, which are precisely the n-complements we constructed in Step 1.
 			\item The other case is:  $B_h=D_1+D_2$, where $D_1,D_2$ are irreducible and $deg(D_i,C)=1$. In this case, we can also apply \ref{prop-comp-surface-fib-glue} to get an $n$-complement $K_W+B_W+R_W$ for $K_W+B_W$, with $R_W\geq 0$ lifting both $n$-complement $K_{D_i}+\Diff(B_W-D_i)+R_{D_i}$ for $K_{D_i}+\Diff(B_W-D_i)$ for $i=1,2$, (since $n$-complements on $D_i$ are constructed in a compatible way) which are precisely the n-complements we constructed in Step 1.
 		\end{enumerate}  Also we can easily see that the $R_W|_D=0$ for any $D$, an irreducible vertical component of $\rddown{B_W}$ by similar reasons as in (a).
 		\end{enumerate}
 	
 \end{enumerate}
Now summing up, we have found $n$, depending only on $\mathfrak{R}$ such that for each $W$, irreducible component of $S$, there is an n-complement $K_W+B_W+R_W$ for $K_W+B_W$ with $R_W\geq 0$ and for each irreducible components $T$ in $\rddown{B_W}$, we have $R_W|_T= R_T$ defined above in step 1.\\\\
Step 3 : We are now done by applying \ref{prop-complements-sdlt} and \ref{prop-gen-lift-comp} again. More precisely, by \ref{prop-complements-sdlt} we can get an $n$-complement for $K_S+B_S$, which will lift to an $n$-complement for $K_X+B$ by \ref{prop-gen-lift-comp}. Pushing forward to $X'$, we get an $n$-complement for $K_{X'}+B'$, which finishes the proof.
\end{proof}

\newpage

\newpage


\begin{thebibliography}{99}
	

\bibitem{Amb}\label{Amb}  {F. Ambro; {\emph{The moduli b-divisor of an lc-trivial fibration.}}
	Compos. Math.
	\textbf{141} (2005), no. 2, 385-403.}

\bibitem{B-BAB}\label{B-BAB}  C. Birkar, {\emph{Singularities of linear systems and boundedness of Fano varieties.}}
arXiv:1609.05543.

\bibitem{B-Fano}\label{B-Fano}
C.~Birkar; \emph{Anti-pluricanonical systems on Fano varieties},  arXiv:1603.05765v2.

\bibitem{B-lc-flips}\label{B-lc-flips}
C.~Birkar,
\emph{Existence of log canonical flips and a special LMMP},
Pub. Math. IHES., 
\textbf{115} (2012), 325-368. 

\bibitem{BH}\label{BH}  {C. Birkar, Z. Hu; {\emph{Polarized pairs, log minimal models, and Zariski decompositions.}} Nagoya Math. J. \textbf{215} (2014), 203-224. }

\bibitem{BCHM}\label{BCHM}
C.~Birkar, P.~Cascini, C.~Hacon and J.~M$^{\rm c}$Kernan;
\emph{Existence of minimal models for varieties of log general type},
J. \ Amer. \ Math. \ Soc. \textbf{23} (2010), no. 2, 405-468.

\bibitem{BZh}\label{BZh}
C.~Birkar and D-Q. Zhang; 
\emph{Effectivity of Iitaka fibrations and pluricanonical systems of polarized pairs}, to appear in Pub. Math. IHES., arXiv:1410.0938v2. 

\bibitem{Bl}\label{Bl}
R.~Blache; 
\emph{The structure of l.c. surfaces of Kodaira dimension zero,}
J. Algebraic
Geom. \textbf{4} (1995) 137-179.

\bibitem{Fuj1} O. Fujino;\emph{abundance theorem for semi log
	canonical threefolds}, duke mathematical journal (2000), Vol. 102, No. 3

\bibitem{Fujino-Gongyo}\label{Fujino-Gongyo}  {O. Fujino, Y. Gongyo; {\emph{On the moduli part of lc-trivial fibrations.}} Ann. Inst. Fourier (Grenoble) \textbf{64} (2014), No. 4, 1721-1735.  }

\bibitem{Fujino-Gongyo-1}  O. Fujino, Y. Gongyo; {\emph{log pluricanonical representations and abundance conjecture}} Compositio Mathematica  April 2014 , pp. 593-620  
\bibitem{Gongyo-1}  Y. Gongyo; {\emph{abundance theorem for numerically trivial log canonical divisors of semi-log canonical pairs}} in Journal of Algebraic Geometry 22(3),  May 2010
\bibitem{Hartshorne}\label{Hartshorne}
R. Hartshorne; \emph{Algebraic geometry.}
Springer (1977).

\bibitem{Hashizume}\label{Hashizume}
K. Hashizume; \emph{A class of singularities of arbitrary pairs and log canonicalization.}
arXiv:1804.06326.


\bibitem{HMX}\label{HMX}  {C. D. Hacon, J. McKernan and C. Xu; {\emph{On the birational automorphisms of varieties of general type.}  } Ann. of Math. (2) \textbf{177} (2013), no. 3, 1077-1111. }


\bibitem{Kol}\label{Kol}
J. Kollar; \emph{Singularities in Minimal Model Program}

\bibitem{kol-old}
D. Abramovich, L.-Y. Fong, J. Kollár, and J. McKernan, \emph{Semi log canonical surface}, Flips and Abundance for Algebraic Threefolds, Astérisque 211, Soc. Math. France, Montrouge, 1992, 139–154.


\bibitem{kollar-mori}\label{kollar-mori}
J.~Koll\'ar and S.~Mori,
Birational geometry of algebraic varieties,
Cambridge Tracts in Math. \textbf{134},
Cambridge Univ.\ Press, 1998.

\bibitem{PSh-II}\label{PSh-II} {Yu. Prokhorov and V.V. Shokurov; {\emph{Towards the second main theorem on complements.}} 
	J.  Algebraic Geometry, \textbf{18} (2009) 151--199.}

\bibitem{PSh-I}\label{PSh-I} {Yu. Prokhorov and V.V. Shokurov; {\emph{The first fundamental Theorem on complements: from global to local.}} (Russian)  Izv. Ross. Akad.  Nauk Ser. Mat.  \textbf{65}  (2001),  no. 6, 99--128;  translation in  Izv. Math.  65  (2001),  no. 6, 1169--1196.}



\bibitem{shokurov-surf-comp}\label{shokurov-surf-comp} {V.V. Shokurov; {\emph{Complements on surfaces.}} 
	Algebraic geometry,  10.  J. Math. Sci. (New York)  
	\textbf{102}  (2000),  no. 2, 3876--3932.}


\bibitem{Shokurov-log-flips}\label{Shokurov-log-flips}  {V.V. Shokurov; {\emph{$3$-fold log flips}},
	With an appendix by Yujiro Kawamata.
	Russian  Acad. \ Sci. \ Izv. \ Math.  \textbf{40}  (1993),  no. 1, 95--202.}

\bibitem{}\label{Zhang}
D-Q. Zhang;
\emph{ Logarithmic Enriques surfaces,} J. Math. Kyoto Univ. \textbf{31} (1991)
419-466.
	
\end{thebibliography}
\end{document}